\documentclass[12pt]{article}
\usepackage{amsfonts}
\usepackage{amsmath,amssymb,amsthm,dsfont,enumerate,bbm}
\usepackage{graphicx}               

\usepackage[textsize=small]{todonotes}       

\newtheorem{theorem}{Theorem}[section]

\newtheorem{lemma}[theorem]{Lemma}

\newtheorem{proposition}[theorem]{Proposition}

\newtheorem{remark}[theorem]{Remark}

\newcommand{\be}{\begin{equation}}
\newcommand{\ee}{\end{equation}}
\newcommand{\nn}{\nonumber}

\def\m{\mu}
\def\n{\eta}

\def\t{\tau}

 \def \Z {{\mathbb Z}}
 \def \R {{\mathbb R}}

 \def \P {{\mathbb P}}
 \def \E {{\mathbb E}}

 \def \cE {\mathcal{E}}

\newcommand{\dint}{{\rm d}}

\newcommand{\Capac}{{\rm Cap}}
\numberwithin{equation}{section}

 \def \1{\mathbbm{1}} 
\def\sfrac#1#2{{\textstyle{#1\over #2}}}

\newcommand{\col}[1]{\textcolor[rgb]{0,0,0}{#1}}
\newcommand{\colo}[1]{\textcolor[rgb]{0,0,0}{#1}}

\title  {Metastability in the reversible inclusion process}

\author
{
 Alessandra Bianchi
 \footnote{
   Universit\`a degli Studi di Padova, Dipartimento di Matematica,
	 Via Trieste, 63, 35121 Padova, Italy.
	 E-mail: {\tt bianchi@math.unipd.it}
 }
 \and
 Sander Dommers
 \footnote{
   Ruhr-Universit\"at Bochum, Fakult\"at f\"ur Mathematik,
   Universit\"atsstra\ss e 150, 44780 Bochum, Germany.
   E-mail: {\tt sander.dommers@ruhr-uni-bochum.de}
 }
 \and
 Cristian Giardin\`a
 \footnote{
   Universit\`a degli Studi di Modena e Reggio Emilia, Dipartimento di Scienze Fisiche, Informatiche e Matematiche,
	 via Campi 213/b, 41125 Modena, Italy.
	 E-mail: {\tt cristian.giardina@unimore.it}
 }
}

\date{\today}

\begin{document}

\maketitle

%

\begin{abstract}
We study the condensation regime of the
finite reversible inclusion process, i.e.\col{,} the inclusion process
on a finite graph $S$ with an underlying random walk
that admits a reversible measure.
We assume that the random walk kernel is irreducible
and its reversible measure takes maximum value on
a subset of vertices $S_{\star}\subseteq S$.
We consider initial conditions corresponding to a single condensate
that is localized on one of those vertices and study
the metastable (or tunneling) dynamics.
We find that, if the random walk restricted to $S_{\star}$
is irreducible, then there exists a single time-scale for the condensate
motion. In this case we compute this typical time-scale and characterize the law of the (properly rescaled)
limiting process. If the restriction of the random walk to $S_{\star}$
has several connected components, a metastability scenario with
multiple time-scales emerges. We prove such \col{a} scenario,
involving two additional time-scales, in a one-dimensional
setting with two metastable states and nearest\col{-}neighb\col{or}
jumps.
\end{abstract}

\section{Introduction}

\colo{The inclusion process is an interacting particle system introduced in the context
of non-equilibrium statistical mechanics, as a dual process of certain diffusion
processes modeling heat conduction and Fourier's law \cite{GiaKurRed, GiaKurRedVaf, GiaRedVaf}.
Besides, it is also related to models in mathematical population
genetics \cite{CarGiaGibRed}, such as the Moran model,
and to models of wealth distribution \cite{CirRedRus}.
In addition to this, the inclusion process is  interesting  in its
own right as an interacting particle system belonging to the class of misanthrope processes
\cite{mis}}.

In the inclusion process, particles jump over a set $S$ of vertices,
thus the total number of particles $N$ is conserved by
the dynamics. The transitions are driven by two competing
contributions to the total \col{jump rate}.
Denoting by $\eta_x$ the particle number at site $x$, and
calling $r: S  \times S \rightarrow \mathbb{R}_+$
the jump rates of a continuous-time irreducible random walk on $S$,
the process is defined by the following rules (see Section \ref{RIP}
for the process generator):
\begin{enumerate}
\item[i)]
firstly, particles move as continuous-time independent random walks:
for a parameter $d_N>0$, each of the $\eta_x$ particles at site $x$ waits
a random time which is the minimum of exponential clocks of
parameters $d_N r(x,z)$ for $z\in S$, and
then jumps to site $y$ with probability $r(x,y)/(\sum_{z\in S} r(x,z))$.
\item[ii)]
secondly, particles jump because of an attractive interaction: each of the $\eta_x$
particles at site $x$ waits a random time which is the minimum of
exponential clocks of parameters $\eta_z r(x,z)$ for $z\in S$,
and then jumps to site $y$ with probability $\eta_y r(x,y)/(\sum_{z\in S} \eta_z r(x,z))$.
\end{enumerate}
Whereas the first contribution leads to {\em spreading} of the particles over the sites of $S$,
due to the second contribution particles have a preference to {\em accumulate} on a few sites.
This is to be compared with the well-known exclusion process, where particles
are subject to hard-core interactions that forbid more than one particle per site.
The inclusion process is a bosonic system, whereas the
exclusion \col{process} is fermionic.

The relative strength of the two contributions is tuned by the parameter sequence $d_N$.
In the long time limit, the two contributions find a compromise
into a (reversible) stationary measure that is shown explicitly in
Section \ref{subsec-condensation}. As \col{long} as $d_N>0$, this measure has mass over all the configuration space.
If the sequence $d_N$ approaches zero
sufficiently fast as $N\to \infty$, then the stationary measure concentrates on a
\col{small} subset of configurations. This is the phenomenon of
{\em condensation} in the inclusion process, first
studied in \cite{GroRedVaf11}. In the condensation regime
essentially all particles accumulate on a single site
of $S_\star\subseteq S$, the set over which the stationary
measure of the random walk takes maximum value. The condensation phenomenon
occurs in several other interacting particle systems \cite{EvaWac},
most notably the \col{z}ero \col{r}ange process \cite{GSS03}.

In this paper w\col{e c}onsider the condensation regime of the
inclusion process and study the dynamics of the condensate.
This problem was previously considered in \cite{GroRedVaf13}.
There, however, the authors assumed a symmetric random
walk kernel that therefore has a uniform reversible measure
on $S$.  Here we consider instead the case of a generic
reversible measure, thus allowing  the possibility that
$S_{\star}\neq S$.  Depending on the properties of the underlying random walk kernel,
the following {\em metastability} scenario with possibly {\em multiple time-scales}
emerges from our analysis.
If the restriction of the random walk kernel to $S_{\star}$ is
still irreducible, then the system has only one time-scale.
However, if such restriction is reducible into
several connected components, then there exist \col{up to}
three time-scales: a first time-scale over which the system
moves within connected components; a second time-scale
to see the jumps between components that are at
graph distance equal to two; a third (even longer) time-scale
for the jumps between components that are at
graph distance larger than two.

Our results include the characterization of the
single time-scale scenario in great generality.
In particular, when the system has only one
time-scale, we allow any geometry
and we are able to derive the rates of the limiting
Markovian dynamics.
We give a rigorous proof of the
multiple time-scale scenario \col{instead} in the one-dimensional
setting, i.e.\col{,} for linear chains with nearest-neighb\col{or}
jumps, whose end-points are the only maximal states of the
reversible measure. In this case we fully characterize the second time-scale
(together with the rates of the limiting dynamics) when
$|S|=3$, and we prove the existence of the third time-scale
when $|S|>3$. We conjecture that the same qualitative behavior
of the motion of the condensate occurs in great generality.

The key ingredient of the proofs of our main results
are {\em potential theory methods}.
We refer in particular to the  potential theoretic approach to metastability,
introduced in a series of papers by Bovier, Eck\col{h}off, Gayrard and Klein \cite{BEGK01}\col{--}\cite{BEGK04},
and to the martingale approach, developed in some recent papers by
Beltr\'{a}n and Landim \cite{BelLan10}\col{--}\cite{BelLan14}.
A general treatment of metastable systems may be found in \cite{OV},
where the pathwise approach to metastability is discussed,
while we refer to \cite{BdH15} for a recent book on the potential theory approach to metastability.

In the next section we introduce the model and state our main results precisely.
 We also give an outline of the proofs.
 In Section~\ref{sec-partition} we analyze the metastable sets.
 The proofs for the three different time-scales are given
 in Sections \ref{sec-cap-first-timescale}--\ref{sec-cap-third-timescale}, respectively.

\section{Model and results}
\subsection{Reversible inclusion process}
\label{RIP}
Consider a set of sites $S$ with $\kappa:=|S|<\infty$ and let $r: S  \times S \rightarrow \mathbb{R}_+$
be the jump rates of a continuous-time irreducible random walk on $S$,
reversible with respect to some probability measure $m$, i.e.,

\be
m(x) r(x,y) = m(y) r(y,x), \qquad {\rm for\ all\ } x,y\in S.
\ee
Without loss of generality, we assume that $r(x,x)=0$ for all $x\in S$.

Of special interest are the sites where $m$ attains its maximum. Hence, we define
\be
M_\star = \max \{m(x) : x\in S \}, \quad S_\star=\{x\in S : m(x)=M_\star\}  {\rm \quad and\quad }  \kappa_\star=|S_\star|.
\ee
and let
\be
m_\star(x)=\frac{m(x)}{M_\star},
\ee
which is a normalization of $m$ such that  $m_\star(x)=1$ if and only if $x\in S_\star$.

For a given underlying random walk we can now give the definition of the {\em reversible inclusion process}
$\{\eta(t) \, : \, t\geq0\}$.
For each $N\geq1$,  the set of configurations $E_N$ correspond to all the possible
arrangements of $N$ particles on $S$, that is
\be
E_{N} = \{\eta\in \mathbb{N}^{S} \,:\, \sum_{x\in S}\eta_x =N \}.
\ee
The component $\eta_x$ of $\n$ has to interpreted as the
number of particles at site $x\in S$.

To specify the possible transitions of the dynamics, for $x,y\in S$, $x\ne y$, and $\n\in E_N$ such that $\n_x >0$,
let us denote by $\n^{x,y}$ the configuration obtained from $\n$ by moving a particle from $x$ to $y$:
\be
(\eta^{x,y})_z = \left\{\begin{array}{ll} \eta_x-1,\qquad &{\rm for\ } z=x, \\ \eta_y+1, &{\rm for\ }z=y\col{,} \\ \eta_z, &{\rm otherwise.}\end{array}\right.
\ee
The {inclusion process} with $N$ particles is then a Markov process $\{\n(t)\,:\,t\geq 0 \}$
on $E_N$ with generator $L_N$, acting on
functions $F:E_N\rightarrow\mathbb{R}$, given by
\be\label{eq-generator-inclusion}
(L_N F)(\eta) = \sum_{x,y \in S}\eta_x \left(d_N+\eta_y\right) r(x,y) [F(\eta^{x,y})-F(\eta)]\;,
\ee
where $\{d_N >0\; : \; N\in\mathbb{N}\}$ is a sequence of positive numbers that \col{is} specified
later.
\subsection{Condensation and metastability}
\label{subsec-condensation}
The inclusion process has a stationary and reversible probability measure $\mu_N(\eta)$,
given by a product measure of negative binomials conditioned to a total number
of particles $N$, i.e.,
\be\label{eq-stationary-measure}
\mu_N(\eta) = \frac{1}{Z_{N,S}} \prod_{x\in S} m_\star(x)^{\eta_x}w_N(\eta_x),
\ee
where
\be
w_N(k) = \frac{\Gamma(k+d_N)}{k!\Gamma(d_N)},
\ee
and
\be
Z_{N,S} = \sum_{\eta\in E_N} \prod_{x\in S} m_\star(x)^{\eta_x}w_N(\eta_x).
\ee
We abbreviate
\be
m_\star^\eta:=\prod_{x\in S} m_\star(x)^{\eta_x} \qquad {\rm and} \qquad w_N(\eta)=\prod_{x\in S} w_N(\eta_x),
\ee
so that~\eqref{eq-stationary-measure} becomes
\be
\mu_N(\eta) = \frac{1}{Z_{N,S}} m_\star^\eta w_N(\eta).
\ee
The stationary measure is unique, because the underlying random walk, and hence also the inclusion process, is irreducible.
It can easily be checked that the measure in~\eqref{eq-stationary-measure} satisfies the detailed
balance, and thus is the reversible measure of the process.

If the parameter $d_N$ scales to zero fast enough in the limit $N\to\infty$, then
the  inclusion  process  shows  condensation, i.e.\col{,}  the invariant measure concentrates on
disjoint sets of configurations (that we shall call {\em metastable sets} or {\em condensates}).
To formalize this idea, let
\be
\mathcal{E}_N^x = \{ \eta : \eta_x=N\}\,,\qquad  x\in S_\star\,.
\ee
Moreover,  for $S_0 \subset S_\star$, define $\mathcal{E}_N(S_0) = \bigcup_{x\in S_0}\mathcal{E}_N^x $
 and let $\Delta = E_N \setminus \mathcal{E}_N(S_\star)$.

 The following result, prove\col{d} in \col{S}ection \ref{sec-partition},
 shows that invariant measure asymptotically concentrates on the sets (in fact singletons)
$\mathcal E_N^x$, $x\in S_\star$, which turn out to be the metastable sets of the process:
\begin{proposition}\label{prop-met-sets}
For $d_N \log N \to 0$ as $N\to\infty$, and  for all $x\in S_\star$,
\be
\lim_{N\to\infty} {\mu_N(\mathcal{E}_N^x)}= \frac{1}{\kappa_\star}.
\ee
As a consequence, $\lim_{N\to\infty} {\mu_N(\Delta)}=0$.
\end{proposition}
The metastability problem we address in this paper is the following.
Assume the process is started from a configuration corresponding to a single condensate.
Then we determine the time-scale(s) over which the condensate moves
and characterize the law of the process describing the motion of the condensate.
\begin{remark}
 Notice that the metastable sets  $\mathcal E_N^x$, $x\in S_\star$, have equal $\mu_N$-measure.
 It may be worth to mention that in this situation  some authors prefer to speak about
  \emph{tunneling  behavior} rather than metastability.
However, this abuse of terminology is currently quite diffuse in the mathematical literature,
and w\col{e j}ust use the word metastability.
\end{remark}

\subsection{Results}
In order to state our findings we introduce some notatio\col{n.}
For a set $A \subset E_N$, let $\t_A$ denote the hitting time of $A$:
\be
\tau_{A} = \inf \{ t\geq 0 \;:\; \eta(t) \in A\}\;.
\ee
Moreover, with the identification $\mathcal E_N^{\star}\equiv \mathcal E_N(S^{\star})$,
let $\eta^{\mathcal{E}_N^{\star}}(t)$ denote the trace process on $\mathcal E_N^{\star}$, i.e.\col{,}
the process obtained from $\n(t)$  by cutting out all time periods where the system
is not in  $\mathcal E_N^{\star}$.
Formally, for all $t\geq 0$,
$\eta^{\mathcal{E}_N^{\star}}(t):= \eta(S_{\mathcal{E}_N^{\star}}(t))$
with $S_{\mathcal{E}_N^{\star}}(t)$ the generalized inverse of the local time $\ell_{\mathcal{E}_N^{\star}}(t)$:
\be
\ell_{\mathcal{E}_N^{\star}}(t) = \int_{0}^t {\bf 1}_{\{\eta(s) \in \mathcal{E}_N^{\star}\}} ds
\,\quad\mbox{and}\quad
S_{\mathcal{E}_N^{\star}}(t) = \sup \{s \ge 0 \,:\, \ell_{\mathcal{E}_N^{\star}}(s) \le t\}\,.
\ee
 Notice that this is still a  Markov process
\colo{(we refer to \cite{BelLan10} for
a proof of this result).
}

Finally, let us define the process
\be\label{eq-metastable-process}
X_N(t) = \psi_N(\eta^{\mathcal{E}_N^{\star} }(t))\col{,}
\ee
where $\psi_N\,: \mathcal{E}_N^{\star} \mapsto S_{\star}$ is given by
\be
\psi_N(\eta) = \sum_{x\in S_{\star}} x \cdot {\bf 1}_{\{\eta \in \mathcal{E}_N^x \}}\,.
\ee

With the above notation, and the usual convention that  $\mathbb{E}_{\eta}(\cdot)$ denotes the expectation
when the process $\eta(t)$ is started from $\n$ at time $t=0$, we prove the following:
\begin{theorem}[First time-scale]\label{thm-first-timescale}
If $d_N \log N \to 0$ as $N\to\infty$, then, for all $x\in S_\star$\col{,}
\begin{itemize}
\item[(i)]
The average time to move  the condensate at $x$ to another site of $S_{\star}$ is given by
\be\label{eq-mean-time}
\mathbb{E}_{\mathcal{E}_N^x} (\tau_{\mathcal{E}_N (S_{\star}\setminus \{x\})}) = \frac{1}{\sum_{y\in S_\star, y \neq x} r(x,y)} \frac{1}{d_N} \left(1 + o(1)\right)\col{.}
\ee
\item[(ii)] Assume that $X_N(0) = x$. Then,
the speeded-up process $\{X_N(t/d_N)\,,\,t\geq 0\}$ converges weakly on the path space
$D(\mathbb R_+,S_{\star})$  to the  Markov process $\{X(t)\,,\, t\geq 0 \}$
on $S_\star$,  with $X(0)=x$ and generator
\be\label{eq-convergence}
{\cal L}f(y) = \sum_{z\in S_\star} r(y,z) [f(z) - f(y)] \,.
\ee
Furthermore, the system spends negligible time outside the metastable states, i.e.\col{,} $\forall \, T >0$
\be\label{eq-0time}
\lim_{N\to\infty} \mathbb{E}_{\mathcal E_N^x} \left[ \int_0^T {\bf 1}_{\{\eta(s/d_N) \in \Delta\}} ds \right] = 0\col{.}
\ee
\end{itemize}
\end{theorem}

\begin{remark}
The weak convergence stated in item (ii) of Theorem \ref{thm-first-timescale}
refers to the pat\col{h s}pace endowed with the Skorokhod topology.
We stress the fact that from this result, together with condition (\ref{eq-0time}),
one can also infer  the weak convergence of the speeded-up projected process $\{\psi_N\left(\n(t/d_N)\right)\,,\,t\geq 0\}$
to the Markov process $\{X(t)\,,\, t\geq 0 \}$  as defined above, though with a topology on the path space, called soft topology,
that is weaker than the Skorokhod one. We refer to \cite{Lan14} for the details.
\end{remark}

From Theorem \ref{thm-first-timescale}, we conclude  that on this first time-scale
 the condensate can only jump between sites in $S_\star$ that are
 connected in the graph induced by the underlying dynamics.
In particular, if $x,y\in S_\star$ are not connected by a path in $S_{\star}$, then the condensate
will not move from $x$ to $y$ on the time-scale $1/d_N$.
Since the  inclusion process is irreducible,
we therefore expect that this movement occurs on a longer time-scale.


We formalize these ideas focusing  on a specific one-dimensional setting.
For an integer $\kappa\geq 2$, let
\be\label{linear-dynamics}
S=[1,\, \kappa]\cap \Z\,, \mbox{ with } \quad
r(x,y)\neq 0\; \mbox{ iff } \; |x-y|=1\,,\quad  S_\star=\{1,\kappa\}\,,
\ee
that is indeed an example of dynamics that is not irreducible when restricted to $S_\star$.

For such systems we have the following, where we say that $d_N$ decays
{\em subexponentially} if, for all $\delta>0$, $\lim_{N\to\infty} d_N e^{\delta N} = \infty$.

\begin{theorem}[Second time-scale]\label{thm-second-timescale}
Consider an underlying random walk as in (\ref{linear-dynamics}), with $\kappa=3$.
Assume that $d_N$ decays subexponentially and $d_N\log N \to 0$ as $N\to\infty$. Then
\begin{itemize}
\item[(i)]
The average time to move  the condensate between the sites of $S_\star$ is given by
\be\label{eq-mean-time2}
\mathbb{E}_{\mathcal E_N^1} (\tau_{\mathcal{E}_N^3})= \mathbb{E}_{\mathcal E_N^3} (\tau_{\mathcal{E}_N^1})  =
\left(\frac{1}{r(1,2)}+\frac{1}{r(3,2)}\right)\cdot\left(1-m_\star(2)\right) \cdot \frac{N}{d_N^2} \left(1 + o(1)\right).
\ee
\item[(ii)] Assume that $X_N(0) = x \in S_\star$. Then,
the speeded-up process $X_N(t N/d_N^2)$ converges weakly on the path space
$D(\mathbb R_+,S_{\star})$  to the  Markov process $\{X(t)\,,\, t\geq 0 \}$
on $S_\star$,  starting at $X(0)=x$ and jumping back and \col{forth} between $x$ and $S_\star\setminus\{ x\}$ at rate
\be\label{eq-convergence2}
\left(\frac{1}{r(1,2)}+\frac{1}{r(3,2)}\right)^{-1}\frac{1}{1-m_\star(2)}\,.
\ee
Furthermore, the system spends negligible time outside the metastable states, i.e.\col{,} $\forall \, T >0$ and $x\in S_\star$,
\be\label{eq-0time2}
\lim_{N\to\infty} \mathbb{E}_{\mathcal E_N^x} \left[ \int_0^T {\bf 1}_{\{\eta(s\cdot N/d_N^2) \in \Delta\}} ds \right] = 0\,.
\ee
\end{itemize}
\end{theorem}

A\col{s  wi}ll be clear from the proof of Theorem \ref{thm-second-timescale}  (see \col{Section }\ref{sec-cap-second-timescale})
the explanation for the presence of this second time-scale
is that, in order to move the condensate between  sites $1$ and $3$,
the system is forced to bring particles through  site $2$.
The presence of particles on sites of $S\setminus S_\star$ is however
an unlikely event, that slows down the motion of the condensate through sites of $S_\star$
and yields a much longer transition time-scale. In this sense, we may consider the sites
of $S\setminus S_\star$ as traps for the dynamics of the system.

Following this idea, the natural further question is about the presence of possibly many time-scales
related to the \emph{length of these traps}, that is to the graph-distance
between disconnected sites of $S_\star$.
We answer this question in the affirmative for linear systems as those defined in (\ref{linear-dynamics}),
proving that an even longer time-scale is required to move the condensate
between the disconnected sites $\{1,\,\kappa\} \in S_\star $ at {\em arbitrary} (but finite) graph-distance greater than $2$.
We have the following:

\begin{theorem}[Third time-scale]\label{thm-third-timescale}
Consider an underlying random walk as in (\ref{linear-dynamics}), with $\kappa> 3$.
Furthermore, assume that $d_N$ decays subexponentially and $d_N\log N \to 0$ as $N\to\infty$.
Then, the average time to move  the condensate between  sites $x,y \in S_\star$, $x\ne y$,
satisfies the bounds
\be\label{eq-3timescale}
C_1 \leq \liminf_{N\to\infty}  \frac{d_N^3}{N^2} \mathbb{E}_{\mathcal E_N^x} (\tau_{\mathcal{E}_N^y})
\leq \limsup_{N\to\infty} \frac{d_N^3}{N^2} \mathbb{E}_{\mathcal E_N^x} (\tau_{\mathcal{E}_N^y}) \leq C_2,
\ee
for some constant $0<C_1,C_2<\infty$.
\end{theorem}

\col{Notice that the upper bound in \eqref{eq-3timescale} excludes the possibility  of the p}resence of longer transition time-scales (or deepe\col{r} traps).
As \col{is} shown in Section  \ref{Lower-bound-third}, the proof
of the upper bound in \eqref{eq-3timescale} (\col{which corresponds to the} lower bound on capacities)
can actually be extended to more general setting beyond the one-dimensional
case. Thus we \col{conjecture} that the inclusion process has always at most
three time-scales for the motion of the condensate, although we  cannot exclude the possibility of the presence of intermediate time-scales.

\subsection{Discussion}
\paragraph{Symmetric inclusion process.}
The paper \cite{GroRedVaf13} proves results similar to those of item (ii) of Theorem~\ref{thm-first-timescale}
in the case where the underlying random walk is symmetric, i.e.\col{,}  $r(x,y)=r(y,x)$, and under the assumption that
$d_N\to 0$ and $d_N N \to \infty$ as $N\to\infty$.
In this case the underlying random walk is reversible with respect to the measure $m_\star \equiv 1$,
so that $S=S_\star$. In particular, all the sites of $S_\star$ belong to the same connected component
 and  the motion of the condensate involves only the first time-scale, of order $1/d_N$.
Let us mention that the results of~\cite{GroRedVaf13} were obtained by completely different techniques, namely
by a  direct scaling and expansion of the generator~\eqref{eq-generator-inclusion}, that was shown
to converge to the generator \eqref{eq-convergence} of the limiting Markov process.

\paragraph{Multiple time-scales.}
Our analysis yields a metastable behavior characterised -- in general --
by multiple time-scales.
Though we prove  the existence  of the second and third time-scale
given in Theorems~\ref{thm-second-timescale} and \ref{thm-third-timescale}
only for the one-dimensional setting in \eqref{linear-dynamics},
we conjecture that the same time-scales show up for general underlying dynamics
and that no further time-scales can occur.
In fact, we expect that the leading mechanism beyond the motion of the condensate can be reduced
to a train of particles moving along single paths between metastable sets.
In this sense, each path can  be seen as a one-dimensional system,
and the results should be proved in a similar way.

However, to formalize this idea one has first to define, for each time-scale,
a new family of metastable sets obtained  by  merging together the metastable states
that are connected on a lower time-scale (a formalization of this merging can for example be found in~\cite{BelLan11}).
Then, one has to show that the reduction to one-dimensional paths is correct, or in other words,
that flows of particles other than that described above, are unlikely to happen.
Because of the complex geometry that may appear in general situations, this may be a rather
difficult task.

 For other systems with multi-scal\col{e} metastable behav\col{ior} see, for example,
 the Blume-Capel model~\cite{CirNar13, LanLem15} and
 the random field Curie-Weiss model~\cite{BiaBov09}.

\paragraph{Comparison with the zero range process.}
The zero range process (ZRP) is a well known interacting particle system that
under suitable hypotheses  displays the condensation phenomenon  (see e.g.~\cite{GSS03, AL09}, and reference therein).
The dynamics of a condensate for the ZRP  has bee\col{n s}tudied in the finite reversible case  in~\cite{BelLan12},
as a first application of the martingale approach to  metastability that was proposed by the same authors
 \cite{BelLan11,BelLan14}. The results have then been generalized to the case of a diverging number of sites
   \cite{BN14,AGL15},  and to the totally asymmetric case \cite{L14}.
  The quite complete picture of metastability obtained in the ZRP,
  allows for a comparison with the results obtained in the reversible inclusion process.
  In both cases:\\
   (i) the  condensate is present only on sites of $S_\star$;\\
 (ii) the metastable sets are equally probable w.r.t.\ the equilibrium measure,
  and thus they are equally stable;\\
   (iii) the energetic barriers that separate
the metastable sets are (at most) logarithmic  with the number of particles,
thus yielding  transition times that are at most polynomial in $N$.

 More interesting are instead differences between the two processes:\\
 (a) the ZRP has only one relevant time-scale, at which the condensate can jump
 directly between any sites in $S_\star$.
 This is due to the fact that the rates of the scaling process on $S_\star$ are given by the capacities
 of the underlying random walk, that are all positive by irreducibility,  thus making irreducible
 also the condensate dynamics;\\
(b) the condensate of the ZRP does not consist of $N$ particles, but only of $N-\ell_N$ particles,
for some $\ell_N$ such that $\ell_N\to\infty$ and $\ell_N/N\to0$ as $N\to\infty$.
 It is exactly due to that small number of particles wandering around the graph,
 that the condensate of the ZRP is able to \col{jump to} all sites of $S_\star$ on the same time-scale.


\subsection{Outline of the proof}\label{subsec-outline}

\colo{As mentioned in the introduction, to prove our theorems we will
use potential theory methods.}
\col{In potential theory},  crucial quantities (at least in the case of reversible dynamics)
are capacities between sets.
Let $D_N$ denote the Dirichlet form associated to the generator $L_N$, that
for ${F\,:\, E_N\mapsto \mathbb R}$, is given by
\be
D_N(F) = \frac12 \sum_{x,y\in S}\sum_{\eta\in E_N} \mu_N(\eta) \eta_x \left(d_N+\eta_y\right) r(x,y) [F(\eta^{x,y})-F(\eta)]^2.
\ee
For two disjoint subsets $A,\,B\subset E_N$, the capacity between $A$ and $B$ can be defined through the
\emph{Dirichlet variational principle}
\be
\Capac_N(A,B) = \inf \{ D_N(F) : F\in \mathcal{F}_N(A,B)\}.
\ee
where
\be
\mathcal{F}_N(A,B) =\{ F: F(\eta)=1 {\rm \ for\ all\ }\eta\in A {\rm \ and\ } F(\eta)=0 {\rm \ for\ all\ }\eta\in B \}.
\ee
The unique minimizer of the Dirichlet principle is the \emph{equilibrium potential}, i.e., the harmonic function
 $h_{A,B}$ that solves the Dirichlet problem
\be\label{eq-Dirichlet-prb}\col{
\left\{
\begin{array}{ll}
  L_N h(\n)=0, & \mbox{ if } \n\notin A\cup B, \\
 h(\n)=1, & \mbox{ if } \n\in A,  \\
 h(\n)=0, & \mbox{ if } \n\in B.
\end{array}
\right.}
\ee
It can be easily checked that
\be\label{eq-potential}
h_{A,B}(\n)= \mathbb P_\n (\t_A< \t_B)\,.
\ee

As pointed out in \cite{BEGK01}\col{--}\cite{BEGK04}, one main fact about capacities in the framework of metastability,
is that they are related to the mean hitting time between sets through the formula
\be\label{eq-hitting-time}
\E_{\nu_{A,B}}(\t_B) = \frac{\mu_N (h_{A,B})}{\Capac_N(A,B)}\,,
\ee
where $\nu_{A,B}$ is a probability measure on $A$ such that, for all $\n \in A$\col{,}
\be
\nu_{A,B}(\n)= \frac{\mu_N(\n)\P_\n(\t_B<\t_{A}^+)}{\Capac_N(A,B)}\col{,}
\ee
and $\t^+_A$ is the return time to $A$, i.e.\col{,}
\be
\t_A^+=\inf\{t> 0\,:\, \n(t)\in A,\, \n(s)\neq \n(0) \mbox{ for some } s\in(0,t)\}\,.
\ee
Notice in particular, that when $A$ is just a singleton, as in the situations we are dealing with,
the measure $\nu_{A,B}$ is just a \col{Dirac delta} over the singleton.
The results stated in Theorems \col{\ref{thm-first-timescale}(i), \ref{thm-second-timescale}(i)} and \ref{thm-third-timescale}
are based o\col{n \eqref{eq-hitting-time}} for $A= \cE_N^x$ and $B= \cE_N(S_\star\setminus \{x\})$.

\col{Capacities also  play an important r\^ole in} \cite{BelLan10},
where  potential theory ideas and martingale methods have been combined
in order to prove the scaling limit of suitably speeded-up processes,
as the one that we have defined in (\ref{eq-metastable-process}).
In our setting, where metastable sets are given by singletons,
the approach of \cite{BelLan10}  to prove the convergence stated in Theorems \col{\ref{thm-first-timescale}(ii)}
and \col{\ref{thm-second-timescale}(ii)}, amounts to verify\col{ing}  the existence of
 a sequence $(\theta_N,\, N\geq 1)$ of positive numbers, tha{t c}orrespond\col{s} to the
chosen time-scale, such that, for any $x,y\in S_\star\,,\,x\neq y$, the following limit exists
\be\label{H0}
p(x,y)\,:=\,\lim_{N\to\infty} \theta_N r_N^{\cE_\star}\left(\cE_N^x,\cE_N^y\right)\,,
\ee
where $r_N^{\cE_\star}\left(\,\cdot\,,\,\cdot\,\right)$
are the  jump rates of the trace process $\n^{\cE_N^\star}(t)$ . 
The set of asymptotic rates $(p(x,y))_{x,y\in S_\star}$ identifies the limiting dynamics.
 By Lemma 6.8. in~\cite{BelLan10},
\be\label{eq-rate-capacity}
\begin{split}
\mu_N(\cE_N^x) r_N^{\cE_\star}\left(\cE_N^x,\cE_N^y\right) =&
\frac 12 \left[ \Capac_N\left(\cE_N^x, \cE_N(S_\star\setminus \{x\})\right)
+\Capac_N\left(\cE_N^y, \cE_N(S_\star\setminus \{y\})\right)\right.\\
&\left.-\Capac_N\left(\cE_N(\left\{x,y\right\}), \cE_N(S_\star\setminus \{x,y\})\right)
\right]\,,
\end{split}
\ee
so that, once more, the main tool to prove our main results  turns out to be the  computation of the  asymptotic capacities.
%

 The computation of the capacities in the first time-scale is performed in Section~\ref{sec-cap-first-timescale},
 while the capacities in the second and in the third time-scale are analysed
 in Sections~\ref{sec-cap-second-timescale} and \ref{sec-cap-third-timescale}, respectively.
 In all the three cases, we first provide a lower bound by restricting the Dirichlet form to a suitable subset of $E_N$ (or flow of configurations).
We then use the obtained insights to construct an approximated equilibrium potential and deduce, via the Dirichlet principle,
a matching upper bound.

In our lower bounds we repeatedly use the following lemma,
which uniformly bounds (parts of) the Dirichlet form from below
by the effective resistance of a linear electrical network.
\begin{lemma}\label{lem-eff-resistance}
Let $R_{i,i+1}>0, i=1,\ldots,k-1$. Then, for any function $F: \{1,\ldots,k\} \to \mathbb{R}$,
\be
\sum_{i=1}^{k-1} R_{i,i+1}[F(i+1)-F(i)]^2 \geq [F(k)-F(1)]^2
\left(\sum_{i=1}^{k-1} \frac{1}{R_{i,i+1}}\right)^{-1}.
\ee
\end{lemma}
\begin{proof}
Define the function
\be
g(i) = \frac{F(i)-F(1)}{F(k)-F(1)},
\ee
so that $g(1)=0$ and $g(k)=1$. Then,
\begin{align}
\sum_{i=1}^{k-1} R_{i,i+1}[F(i+1)-F(i)]^2  &=
[F(k)-F(1)]^2 \sum_{i=1}^{k-1} R_{i,i+1}[g(i+1)-g(i)]^2 \nn\\
&\geq  [F(k)-F(1)]^2 \inf_{h: h(1)=0,\atop h(k)=1} \sum_{i=1}^{k-1} R_{i,i+1}[h(i+1)-h(i)]^2 \nn\\
&= [F(k)-F(1)]^2 \left(\sum_{i=1}^{k-1} \frac{1}{R_{i,i+1}}\right)^{-1},
\end{align}
where the last equality follows using the series law for the effective capacity of a linear chain
(see, e.g., \cite{LPW}).
\end{proof}

\section{Metastable sets}\label{sec-partition}
In this section we study the partition function $Z_{N,S}$ and
characterize its asymptotic behavior in the limit $N\to\infty$.
This result \col{is} used
to prove that the configurations
in $\Delta= E_N\setminus \cE_N^\star$  are very unlikely in equilibrium and that $\mathcal{E}_N^x, x\in S_\star$
are the metastable sets (Proposition \ref{prop-met-sets}).
That in turn is the main ingredient for the proof of
\eqref{eq-0time} and \eqref{eq-0time2}
\col{in} Theorem\col{s} \ref{thm-first-timescale} and \ref{thm-second-timescale}, respectively.

We start analyzing the weight function $w_N(\ell)$.
\begin{lemma}\label{lem-asymp-wN}
For $d_N \log N \to 0$ as $N\to\infty$, and $0\leq k <N$,
\be
\lim_{N\to\infty} \frac{(d_N+k)w_N(k)}{d_N}  = \lim_{N\to\infty} \frac{(k+1)w_N(k+1)}{d_N}  = 1.
\ee
\end{lemma}
\begin{proof}
First note that
\be
(d_N+k)w_N(k) = (d_N+k) \frac{\Gamma(k+d_N)}{k!\Gamma(d_N)} = \frac{(k+1)\Gamma(k+1+d_N)}{(k+1)!\Gamma(d_N)} = (k+1)w_N(k+1),
\ee
so that indeed the two limits are the same.

We rewrite
\be
\frac{(k+1)w_N(k+1)}{d_N} = \frac{1}{d_N} \frac{(k+1)\Gamma(k+1+d_N)}{(k+1)!\Gamma(d_N)} = \frac{1}{\Gamma(d_N+1)} \frac{\Gamma(k+1+d_N)}{\Gamma(k+1)}.
\ee
Clearly,
\be
\lim_{N\to\infty} \frac{1}{\Gamma(d_N+1)}=\frac{1}{\Gamma(1)}=1,
\ee
and
\be
\frac{\Gamma(k+1+d_N)}{\Gamma(k+1)} \geq 1.
\ee
The upper bound follows from Wendel's inequality~\cite{Wen48}:
\be
\frac{\Gamma(k+1+d_N)}{\Gamma(k+1)} \leq (k+1)^{d_N} \leq N^{d_N} = e^{d_N \log N},
\ee
which indeed converges to $1$ by our assumption on $d_N$.
\end{proof}

We can now compute the limiting behavior of the partition function:
\begin{proposition}\label{prop-limZN}
For $d_N \log N \to 0$ as $N\to\infty$,
\be
\lim_{N\to\infty} \frac{N}{d_N} Z_{N,S} = \kappa_\star.
\ee
\end{proposition}
\begin{proof}
Since $Z_{N,S}$ includes the $\kappa_\star$ configurations where all particles are on one of the sites in $S_\star$, it is clear that
\be\label{eq-lbZn}
Z_{N,S} \geq \kappa_\star w_N(N) = \kappa_\star \frac{d_N}{N}(1+o(1)),
\ee
by Lemma~\ref{lem-asymp-wN}.

 To prove the upper bound we proceed by induction.
 We label the sites by $1,\ldots,\kappa$,
 and let $E_{n,k}$ be the set of configurations
 with $n$ particles on the first $k$ sites,
 i.e., $E_{n,k} = E_{n,\{1,\ldots,k\}}$.
 Let us define, for $1\leq n \leq N$,
\be
Z_{n,k} = \sum_{\eta \in E_{n,k}} m_\star^{\eta} w_N(\eta).
\ee
By induction over $k$, we aim to prove that, for all $1\leq n \leq N$ , $1\leq k\leq\kappa$, and $N$ large enough,
\be\label{eq-inductionboundZnk}
Z_{n,k} \leq \frac{d_N (1+o(1))}{n}\sum_{s=1}^k m_\star(s)^n + C_k \frac{d_N^2 \log n}{n},
\ee
where $C_k<\infty$ is a constant that only depends on $k$ and may change from line to line.

We start the induction with $k=1$, for which clearly, by Lemma~\ref{lem-asymp-wN},
\be
Z_{n,1} = m_\star(1)^n w_N(n) = \frac{d_N (1+o(1))}{n} m_\star(1)^n\,.
\ee
Assume  that (\ref{eq-inductionboundZnk}) holds true for $k-1$ and for all $1\leq n \leq N$.
Then, using the induction hypothesis and Lemma~\ref{lem-asymp-wN},
\begin{align}
Z_{n,k} &= m_\star(k)^n w_N(n) + Z_{n,k-1}+\sum_{\ell=1}^{n-1} m_\star(k)^\ell w_N(\ell) Z_{n-\ell,k-1} \nn\\
&\leq  \frac{d_N (1+o(1))}{n}  \left(m_\star(k)^n + \sum_{s=1}^{k-1} m_\star(s)^n\right) +C_{k-1}\frac{d_N^2 \log n}{n}\\
&\qquad + \sum_{\ell=1}^{n-1} m_\star(k)^\ell  \frac{d_N(1+o(1))}{\ell} \left(\left(\sum_{s=1}^{k-1} m_\star(s)^{n-\ell}\right)
\frac{d_N(1+o(1))}{(n-\ell)}+ C_{k-1} \frac{d_N^2 \log(n-\ell)}{n-\ell}\right).\nn
\end{align}
Using that $m_\star(k)\leq1$ and $d_N \log(n-\ell) = o(1)$ by assumption, and  for $N$ large enough,
the sum in $\ell$ can be bounded from above by
\be\label{eq-bderrorsum}
C_k d_N^2  \sum_{\ell=1}^{n-1}  \frac{1}{\ell(n-\ell)} = 2C_k d_N^2 \sum_{\ell=1}^{n/2} 
\frac{1}{\ell(n-\ell)} = 2C_k \frac{d_N^2}{n} \sum_{\ell=1}^{n/2} \frac{1}{\ell(1-\ell/n)} \,.
\ee
Since $\ell\leq n/2$ we have that $(1-\ell/n) \geq \frac12$. Hence, we can bound~\eqref{eq-bderrorsum} from above by
\be
4 C_k \frac{d_N^2}{n} \sum_{\ell=1}^{n/2} \frac{1}{\ell} \leq 4 C_k \frac{d_N^2 \log n}{n} .
\ee
This proves the induction step. Thus,
\begin{align}
Z_{N,\kappa} &\leq \frac{d_N (1+o(1))}{N}\sum_{s=1}^\kappa m_\star(s)^N + C_\kappa \frac{d_N^2 \log N}{N}\\
&= \kappa_\star \frac{d_N (1+o(1))}{N} + \frac{d_N (1+o(1))}{N} \left(\left(\sum_{s\notin S_\star} m_\star(s)^N\right) + C_\kappa d_N \log N\right) \nn\\
&= \kappa_\star  \frac{d_N}{N}(1+o(1)).
\end{align}
The proposition follows by combining this upper bound with the lower bound in~\eqref{eq-lbZn}.
\end{proof}

Combining these results, Proposition \ref{prop-met-sets} follows trivially:
\begin{proof}[Proof of Proposition \ref{prop-met-sets}]
For all $x\in S_\star$\col{, b}y Lemma \ref{lem-asymp-wN} and Proposition \ref{prop-limZN},
\be
\lim_{N\to\infty}\m_N(\cE_N^x) =\lim_{N\to\infty} \frac{w_N(N)}{Z_{N,S}}=
\lim_{N\to\infty}\frac{N\cdot w_N(N)}{d_N}\frac{d_N}{N \cdot Z_{N,S}}
= \frac{1}{\kappa^*}\,.
\ee
As a consequence,
\be
\lim_{N\to\infty}\m_N(\Delta) = 1- \sum_{x\in S_\star}\lim_{N\to\infty}\m_N(\cE_N^x)=0\,.
\ee
\end{proof}

\section{Dynamics of the condensate on the first  time-scale}\label{sec-cap-first-timescale}
In this section we analyze capacities on the time-scale $1/d_N$ and
prove  Theorem \ref{thm-first-timescale}.
 We prove the lower bound on capacities in Section~\ref{sec-cap-scale1-lb},
 the upper bound in Section~\ref{sec-cap-scale1-ub},
 and we give the proof of Theorem \ref{thm-first-timescale} in Section~\ref{sec-Proof-thm1}.

\subsection{Lower bound on capacities}\label{sec-cap-scale1-lb}
\begin{proposition} \label{prop-capacities-between-metastable-sets-lb}
For a nonempty subset $S_\star^1 \subsetneq S_\star$,
 let $S_\star^2 = S_\star \setminus S_\star^1$. Then, for $d_N \log N \to 0$ as $N\to\infty$,
\be
\col{\liminf_{N\rightarrow\infty}} \frac{1}{d_N} \Capac_N \left(\mathcal{E}_N(S_\star^1),\mathcal{E}_N(S_\star^2)\right) \geq \frac{1}{\kappa_\star} \sum_{x \in S_\star^1}\sum_{y \in S_\star^2}r(x,y).
\ee
\end{proposition}
\begin{proof}
Let
\be
A_N^{x,y} = \left\{\eta\in E_N \,:\, \eta_x+\eta_y = N \right\}.
\ee
Fix a function $F\in \mathcal{F}_N(\mathcal{E}_N(S_\star^1),\mathcal{E}_N(S_\star^2))$. Then,
\begin{align}
D_N(F) &= \frac12 \sum_{x,y\in S}\sum_{\eta \in E_N} \mu_N(\eta) \eta_x \left(d_N+\eta_y\right) r(x,y) [F(\eta^{x,y})-F(\eta)]^2 \nn\\
& \geq \frac12 \sum_{x\in S_\star^1}\sum_{y\in S_\star^2} \sum_{\eta\in A_N^{x,y}}\mu_N(\eta) \eta_x \left(d_N+\eta_y\right) r(x,y) [F(\eta^{x,y})-F(\eta)]^2 \nn\\
& \qquad + \frac12 \sum_{y\in S_\star^2}\sum_{x\in S_\star^1} \sum_{\eta\in A_N^{x,y}}\mu_N(\eta) \eta_y \left(d_N+\eta_x\right) r(y,x) [F(\eta^{y,x})-F(\eta)]^2 \nn\\
&  = \sum_{x\in S_\star^1}\sum_{y\in S_\star^2}  r(x,y) \sum_{\eta\in A_N^{x,y}}\mu_N(\eta) \eta_x \left(d_N+\eta_y\right)  [F(\eta^{x,y})-F(\eta)]^2,
\end{align}
by reversibility.
Note that  the set $A_N^{x,y}$  can be parameterized by the number of particles at $x$,
and is thus a one-dimensional set.
For a fixed couple $x,y\in S_\star$, let $G$ be the restriction of $F$ to the set $A_N^{x,y}$, i.e.,
for $\n\in A_N^{x,y}$ such that $\eta_x = \ell$,  define $G(\ell):= F(\eta)$.
Then we can rewrite
\begin{align}\label{eq-diri1}
\sum_{\eta\in A_N^{x,y}}\mu_N(\eta) &\eta_x \left(d_N+\eta_y\right)  [F(\eta^{x,y})-F(\eta)]^2 \nn\\
&= \sum_{\ell =1}^N \frac{w_N(\ell)w_N(N-\ell)}{Z_{N,S}} \ell \left(d_N+N-\ell\right)  [G(\ell-1)-G(\ell)]^2,
\end{align}
where we used that $m_\star(x)=m_\star(y)=1$.

Using Lemma~\ref{lem-asymp-wN} for all $1\leq\ell\leq N$,
\be\label{eq-bounddnsq}
w_N(\ell)w_N(N-\ell) \ell \left(d_N+N-\ell\right) = d_N^2 (1+o(1)),
\ee
so that~\eqref{eq-diri1} equals
\be
\frac{d_N^2 (1+o(1))}{Z_{N,S}} \sum_{\ell =1}^N [G(\ell-1)-G(\ell)]^2.
\ee
Note that $G(0)=0$ and $G(N)=1$, so that it follows from Lemma~\ref{lem-eff-resistance} that
\be
\sum_{\ell =1}^N [G(\ell-1)-G(\ell)]^2\geq \frac{1}{N}.
\ee
Hence,
\be
D_N(F) \geq \frac{d_N^2}{N Z_{N,S}} \sum_{x\in S_\star^1}\sum_{y\in S_\star^2} r(x,y),
\ee
and the proposition follows from Proposition~\ref{prop-limZN}.
\end{proof}


\subsection{Upper bound on capacities}\label{sec-cap-scale1-ub}
\begin{proposition} \label{prop-capacities-between-metastable-sets-ub}
For a nonempty subset $S_\star^1 \subsetneq S_\star$,
 let $S_\star^2 = S_\star \setminus S_\star^1$. Then, for $d_N \log N \to 0$ as $N\to\infty$,
\be
\col{\limsup_{N\rightarrow\infty}} \frac{1}{d_N} \Capac_N \left(\mathcal{E}_N(S_\star^1),\mathcal{E}_N(S_\star^2)\right) \leq \frac{1}{\kappa_\star} \sum_{x \in S_\star^1}\sum_{y \in S_\star^2}r(x,y).
\ee
\end{proposition}
The strategy of the proof is to provide a suitable test function $F\in \mathcal{F}_N(\mathcal{E}(S_\star^1),\mathcal{E}(S_\star^2))$
to plug in the Dirichlet principle
\be\label{eq-capasinf}
\Capac_N(\mathcal{E}(S_\star^1),\mathcal{E}(S_\star^2)) = \inf \{ D_N(F) : F\in \mathcal{F}_N(\mathcal{E}(S_\star^1),\mathcal{E}(S_\star^2))\}\,.
\ee
We first describe how we construct the test function, and then study the corresponding Dirichlet form
by splitting it into several parts, and analyzing each of them separately.
We conclude the section collecting all the results and providing the proof of the above proposition.

\paragraph*{Construction of the test function.}
Inspired by the lower bound derived in the previous section, we want $F(\eta)$ to be approximately equal to $G^*(\eta_x)$ , where $G^*$
is the minimizer of ${\sum_{\ell =1}^N [G(\ell-1)-G(\ell)]^2}$, which is given by
\be
G^*(\ell)  = \frac{\ell}{N}.
\ee
To avoid difficulties for small and large values of $\eta_x$,
we choose an arbitrary small $\varepsilon>0$ and set the function equal to $0$
if $\eta_x/N \leq \varepsilon$, and equal to $1$ if $\eta_x/N \geq 1-\varepsilon$.

For values $ \eta_x/N \in (\varepsilon,1-\varepsilon)$ ,
we approximate $G^*(\n_x)$ with a smooth function $\phi_\varepsilon(\eta_x/N)$
defined as in~\cite{BelLan12}.
That is, $\phi_\varepsilon:[0,1]\to[0,1]$ is a smooth nondecreasing function satisfying
${\phi_\varepsilon(t)+\phi_\varepsilon(1-t)=1}$ for all $t\in[0,1]$,
$\phi_\varepsilon(t)=0$ for $t\leq\varepsilon$,
$\phi_\varepsilon(t)=1$ for $t\geq 1-\varepsilon$,
and $\phi_\varepsilon'(t)\leq 1+\sqrt{\varepsilon}$ for all $t\in[0,1]$.
Such a function exists since $(1+\sqrt{\varepsilon})$ times
the length of the interval $[\varepsilon,1-\varepsilon]$ is strictly bigger than $1$ for $\varepsilon$ small enough.

All together, for any $x\in S$, we define the functions $F_x: E_N\mapsto \R$ as
\be
F_x(\eta) = \phi_\varepsilon(\eta_x/N),
\ee
and similarly, for $S^1 \subset S$, the functions $F_{S^1}: E_N\mapsto \R$ as
\be
F_{S^1}(\eta) = \sum_{x \in S^1} F_x(\eta).
\ee

\paragraph*{Split of the Dirichlet form.}
To split the Dirichlet form, define, for a set $A\subseteq E_N$,
\be
D_N(F,A) = \frac12 \sum_{\eta\in A} \mu_N(\eta) \sum_{z,w\in S} \eta_z (d_N+\eta_w)r(z,w) [F(\eta^{z,w})-F(\eta)]^2.
\ee
Also define
\be
A_N^x = \bigcup_{y\in S\setminus \{x\}} A_N^{x,y}.
\ee
We can then write
\be\label{eq-splitD1}
D_N(F_{S^1})=D_N(F_{S^1}, E_N) = D_N(F_{S^1}, \bigcup_{x\in S^1} A_N^x) + D_N(F_{S^1}, E_N \setminus\bigcup_{x\in S^1} A_N^x)\col{.}
\ee
By definition
\be
D_N(F_{S^1}, \bigcup_{x\in S^1} A_N^x) = D_N(F_{S^1}, \bigcup_{x\in S^1} \bigcup_{y\in S\setminus\{x\}}A_N^{x,y}).
\ee
If $\{x_1,y_1\} \neq \{x_2,y_2\}$ and $\eta\in A_N^{x_1,y_1} \cap A_N^{x_2,y_2}$, then either $x_1=x_2$ and $\eta_{x_1}=N$, or $y_1=y_2$ and $\eta_{y_1}=N$. In both cases, $F_{S^1}(\eta^{z,w}) = F_{S^1}(\eta)$ for all $z,w\in S$ because of the definition of $\phi_\varepsilon$. Therefore, we can write
\be\label{eq-splitD2}
D_N(F_{S^1}, \bigcup_{x\in S^1} A_N^x) =\sum_{x\in S^1} \sum_{y \in S^2} D_N(F_{S^1}, A_N^{x,y})  + \frac12 \sum_{x,y\in S^1} D_N(F_{S^1}, A_N^{x,y}),
\ee
where $S^2 = S\setminus S^1$.
\paragraph*{Dirichlet form inside tubes.}
The main contribution to the Dirichlet form comes from configurations
inside tubes between sites $x\in S^1, y\in S^2$, as the next lemma shows.
\begin{lemma}\label{lem-insidetubesS1S2}
Let $x\in S^1$ and $y\in S^2$. Then, for $d_N \log N \to 0$ as $N\to\infty$,
\be
\lim_{\varepsilon\to0}\col{\limsup_{N\to\infty}}\frac{1}{d_N} D_N(F_{S^1}, A_N^{x,y}) \leq \frac{1}{\kappa_\star} r(x,y) \1 \{x,y\in S_\star\}.
\ee
\end{lemma}
\begin{proof}
Note that if $\eta \in A_N^{x,y}$, then for $v\in S^1\setminus\{x\}$ we have that $F_v(\eta)=F_v(\eta^{z,w})=0$, since $\eta_v,\eta_v^{z,w}\leq 1<\varepsilon N$. Hence,
\be
D_N(F_{S^1}, A_N^{x,y}) =  D_N(F_x, A_N^{x,y}).
\ee
Note also that for configurations such that $\eta_x<\varepsilon N$, or $\eta_x>(1-\varepsilon)N$ ,
we have that ${F_x(\eta^{z,w})=F_x(\eta)}$.
Hence, we can restrict the sum to configurations $\n$ such that  $\varepsilon N\leq \eta_x \leq (1-\varepsilon)N$ and get
\begin{align}
D_N(F_x, A_N^{x,y}) &= \frac{1}{2 Z_{N,S}} \sum_{j=\varepsilon N}^{(1-\varepsilon)N} m_\star(x)^{j} m_\star(y)^{N-j}w_N(j) w_N(N-j) \nn\\
&\qquad\ \qquad\ \qquad \ \qquad \ \qquad \sum_{z,w \in S} \eta_z (d_N+\eta_w) r(z,w) [F_x(\eta^{z,w})-F_x(\eta)]^2.
\end{align}
Since $F_x(\eta)=\phi_\varepsilon(\eta_x)$ does not change if the number of particles on $x$ stays the same,
we can further rewrite this, also using reversibility, as
\begin{align}
D_N(F_x, A_N^{x,y})& =\frac{1}{Z_{N,S}} \sum_{j=\varepsilon N}^{(1-\varepsilon)N} m_\star(x)^{j} m_\star(y)^{N-j} w_N(j) w_N(N-j) \nn\\
&\qquad\ \qquad\ \qquad \ \qquad \ \qquad \biggl\{  j(d_N+N-j) r(x,y)  \Bigl[\phi_\varepsilon\bigl(\sfrac{j-1}{N}\bigr)-\phi_\varepsilon\bigl(\sfrac{j}{N}\bigr)\Bigr]^2 \\
&\qquad\ \qquad\ \qquad \ \qquad \ \qquad\ \qquad +\sum_{z \in S\setminus\{x,y\}} j d_N r(x,z)
\Bigl[\phi_\varepsilon\bigl(\sfrac{j-1}{N}\bigr)-\phi_\varepsilon\bigl(\sfrac{j}{N}\bigr)\Bigr]^2 \biggr\}\,.\nn
\end{align}
Because of the bound on $\phi_\varepsilon'(t)$, we have that
\be\label{eq-boundphiprimeN}
\Bigl| \phi_\varepsilon\bigl(\sfrac{j+1}{N}\bigr)-\phi_\varepsilon\bigl(\sfrac{j}{N}\bigr) \Bigr| \leq \frac{1+\sqrt{\varepsilon}}{N}.
\ee
Thus, also using Lemma~\ref{lem-asymp-wN},
\begin{align}
D_N(F_x, A_N^{x,y})&\leq  \frac{d_N^2(1+o(1))}{N^2 Z_{N,S}} (1+\sqrt{\varepsilon})^2 m_\star(x)^{\varepsilon N} m_\star(y)^{\varepsilon N} \sum_{j=\varepsilon N}^{(1-\varepsilon)N} \biggl\{r(x,y) + \sum_{z \in S\setminus\{x,y\}} \frac{ d_N}{N-j} r(x,z) \biggr\} \nn\\
&= \frac{d_N(1+o(1))}{ \kappa_\star} (1+\sqrt{\varepsilon})^2 (1-2\varepsilon) m_\star(x)^{\varepsilon N} m_\star(y)^{\varepsilon N} r(x,y),
\end{align}
where  in the second equality we used Proposition~\ref{prop-limZN}.
Hence,
\be
\col{\limsup_{N\to\infty}} \frac{1}{d_N} D_N(F_x, A_N^{x,y}) \col{\leq} \frac{r(x,y)}{\kappa_\star} (1+\sqrt{\varepsilon})^2(1-2\varepsilon) \1 \{x,y\in S_\star\},
\ee
and the lemma follows by taking the limit $\varepsilon\to0$.
\end{proof}
The contribution to the Dirichlet form
coming from configurations inside a tube between sites $x,y \in S_1$
is negligible, as the next lemma shows.
\begin{lemma}\label{lem-insidetubesS1}
Let $x,y\in S^1$. Then, for $d_N \log N \to 0$ as $N\to\infty$,
\be
\lim_{N\to\infty}\frac{1}{d_N} D_N(F_{S^1}, A_N^{x,y}) =0.
\ee
\end{lemma}
\begin{proof}
Again, note that if $\eta \in A_N^{x,y}$, then for $v\in S^1\setminus\{x,y\}$
we have that $F_v(\eta)=F_v(\eta^{z,w})=0$, since $\eta_v,\eta_v^{z,w}\leq 1<\varepsilon N$.
Thus,
\be
D_N(F_{S^1}, A_N^{x,y}) =  D_N(F_x+F_y, A_N^{x,y}).
\ee
If a particle moves from $x$ to $y$, or viceversa, the total number of particles on sites $x$ and $y$ stays equal to $N$ and hence
\be
F_x(\eta) + F_y(\eta) =F_x(\eta^{x,y}) + F_y(\eta^{x,y})=F_x(\eta^{y,x}) + F_y(\eta^{y,x}) =  1,
\ee
since  by definition, $\phi_\varepsilon(x)+\phi_\varepsilon(1-x)=1$ for all $x\in[0,1]$.
We can  use again that $F_v$ does not change if the number of particles on $v$ stays the same,
 and restrict the sum to configurations with $\varepsilon N \leq \eta_x \leq (1-\varepsilon)N$. Thus
\begin{align}
D_N(F_{S^1}, A_N^{x,y}) &= \frac12 \sum_{\eta\in A_N^{x,y}}\mu_N(\eta) \sum_{z \in S\setminus\{x,y\}} \biggl\{ \eta_x d_N r(x,z) \Bigl[F_x(\eta^{x,z})-F_x(\eta)\Bigr]^2 \nn\\
&\qquad\ \qquad \ \qquad \ \qquad\ \qquad\ \qquad\ \qquad+  \eta_y d_N r(y,z) \Bigl[F_y(\eta^{y,z})-F_y(\eta)\Bigr]^2\biggr\}\nn\\
& = \frac{d_N}{2 Z_{N,S}} \sum_{j=\varepsilon N}^{(1-\varepsilon)N} m_\star(x)^{j} m_\star(y)^{N-j} w_N(j) w_N(N-j)\biggl\{ j r(x,z) \Bigl[\phi_\varepsilon\bigl(\frac{j-1}{n}\bigr)-\phi_\varepsilon\bigl(\frac{j}{n}\bigr)\Bigr]^2 \nn\\
&\qquad\ \qquad\ \qquad\ \qquad+  (N-j)  r(y,z) \Bigl[\phi_\varepsilon\bigl(\frac{N-j-1}{N}\bigr)-\phi_\varepsilon\bigl(\frac{N-j}{N}\bigr)\Bigr]^2\biggr\}.
\end{align}
Using Lemma~\ref{lem-asymp-wN},~\eqref{eq-boundphiprimeN} and $m_\star(x),m_\star(y)\leq 1$, we can bound
\begin{align}
D_N(F_{S^1}, A_N^{x,y})& \leq  \frac{d_N^3 (1+o(1))}{2 N^2 Z_{N,S}} (1+\sqrt{\varepsilon})^2 \sum_{j=\varepsilon N}^{(1-\varepsilon)N}  \biggl\{ \frac{ r(x,z) }{N-j}+ \frac{ r(y,z)}{j} \biggr\} \nn\\
& =  \frac{d_N (1+o(1))}{2 \kappa_\star} (1+\sqrt{\varepsilon})^2 (1-2\varepsilon) o(1).
\end{align}
Then finally,
\be
\lim_{N\to\infty}\frac{1}{d_N} D_N(F_{S^1}, A_N^{x,y}) = \lim_{N\to\infty} d_N o(1) = 0.
\ee
\end{proof}

\paragraph*{Dirichlet form outside tubes}
We finally show in the next lemma, that the configurations outside the collections of tubes $A_N^z$
 gives a negligible contribution to the Dirichlet form.
\begin{lemma}\label{lem-outsidetubes}
For $d_N \log N \to 0$ as $N\to\infty$,
\be
\lim_{N\to\infty}\frac{1}{d_N} D_N(F_{S^1}, E_N \setminus\bigcup_{x\in S^1} A_N^x) =0.
\ee
\end{lemma}
\begin{proof}
As in~\cite{BelLan12}, by the Cauchy-Schwarz inequality we get
\be
[F_{S^1}(\eta^{z,w})-F_{S^1}(\eta)]^2 = \biggl[\sum_{x \in S^1} [F_x(\eta^{z,w})-F_x(\eta)]\biggr]^2 \leq |S^1| \sum_{x \in S^1} [F_x(\eta^{z,w})-F_x(\eta)]^2,
\ee
and then
\be\label{eq-CZoutsidetube}
D_N(F_{S^1}, E_N \setminus\bigcup_{z\in S^1} A_N^z)  \leq |S^1| \sum_{x\in S^1}D_N(F_x, E_N \setminus\bigcup_{z\in S^1} A_N^z) \leq  |S^1| \sum_{x\in S^1}D_N(F_x, E_N \setminus A_N^x).
\ee
Again, we can restrict the sum to configurations with $\varepsilon N\leq \eta_x \leq (1-\varepsilon)N$.
Furthermore, if $\eta\in E_N \setminus A_N^x$ and $\eta_x=j$, all sites besides $x$ have at most $N-j-1$ particles, and thus
\be
D_N(F_x, E_N \setminus A_N^x) = \frac12 \sum_{j=\varepsilon N}^{(1-\varepsilon)N} \sum_{\substack{\eta: \eta_x=j \\
 \forall y\neq x:\,\eta_y \leq N-j-1}} \!\!\!\mu_N(\eta) \sum_{z,w\in S} \eta_z (d_N+\eta_w)r(z,w) [F_x(\eta^{z,w})-F_x(\eta)]^2.
\ee
Note that if $z,w\neq x$, then $F_x(\eta^{z,w})=F_x(\eta)$, since $F_x$ only depends on the number of particles on~$x$. Hence,
\begin{align}
D_N&(F_x, E_N \setminus A_N^x) =  \frac12 \sum_{j=\varepsilon N}^{(1-\varepsilon)N}
 \sum_{\substack{\eta\in E_N: \eta_x=j \\  \forall y\neq x:\,\eta_y \leq N-j-1}} \mu_N(\eta) \,\nn\\
& \quad\quad  \sum_{y\in S\setminus \{x\}}
\biggl\{ \eta_x (d_N+\eta_y)r(x,y)
[F_x(\eta^{x,y})-F_x(\eta)]^2+ \eta_y (d_N+\eta_x)r(y,x) [F_x(\eta^{y,x})-F_x(\eta)]^2\biggr\}\nn\\
&=  \frac1{2Z_{N,S}} \sum_{j=\varepsilon N}^{(1-\varepsilon)N} m_\star(x)^j w_N(j)
 \sum_{\substack{\eta\in E_N: \eta_x=j \\  \forall y\neq x:\,\eta_y \leq N-j-1}}
 \prod_{y\in S\setminus \{x\}} \Bigl(m_\star(y)^{\eta_y} w_N(\eta_y)\Bigr)\,\\
&\quad \sum_{y\in S\setminus\{x\}} \biggl\{ j (d_N+\eta_y)r(x,y)
\Bigl[\phi_\varepsilon\bigl(\sfrac{j-1}{N}\bigr)-\phi_\varepsilon\bigl(\sfrac{j}{N}\bigr)\Bigr]^2+
\eta_y (d_N+j)r(y,x) \Bigl[\phi_\varepsilon\bigl(\sfrac{j+1}{N}\bigr)-\phi_\varepsilon\bigl(\sfrac{j}{N}\bigr)\Bigr]^2\biggr\}. \nn
\end{align}
Since $|S|<\infty$, we can bound $r(x,y),r(y,x) \leq \max_{z,w\in S} r(z,w) =: R$
and $m_\star(x)\leq 1$, and also bound $\max\{j(d_N+N-j),(N-j)(d_N+j) \}\leq j (N-j) (1+o(1))$.
Combining this with~\eqref{eq-boundphiprimeN}, we get
\begin{align}
D_N&(F_x, E_N \setminus A_N^x)  \\
&\leq R (\kappa-1)\frac{(1+\sqrt{\varepsilon})^2}{N^2Z_{N,S}}
\sum_{j=\varepsilon N}^{(1-\varepsilon)N}  w_N(j) j (N-j)(1+o(1))
\sum_{\substack{\eta\in E_N: \eta_x=j \\ \forall y\neq x:\,\eta_y \leq N-j-1}}
\prod_{y\in S\setminus \{x\}} \Bigl(m_\star(y)^{\eta_y} w_N(\eta_y)\Bigr)\col{.} \nn
\end{align}
Notice that  the last sum can be written as
\be
Z_{N-j,S\setminus\{x\}} - \sum_{y\in S\setminus \{x\}} m_\star(y)^{N-j} w_N(N-j)\leq \frac{d_N}{N-j}o(1),
\ee
where the inequality follows from~\eqref{eq-inductionboundZnk}. Hence, also using Lemma~\ref{lem-asymp-wN} and Proposition~\ref{prop-limZN},
\begin{align}
D_N(F_x, E_N \setminus A_N^x)& \leq R (\kappa-1)\frac{(1+\sqrt{\varepsilon})^2}{N^2 Z_{N,S}} (1-2\varepsilon)N d_N^2 (1+o(1)) o(1) \nn\\
&= R \frac{(\kappa-1)}{\kappa_\star}(1+\sqrt{\varepsilon})^2 (1-2\varepsilon) d_N (1+o(1)) o(1),
\end{align}
from which it follows that
\be
\frac{1}{d_N} D_N(F_x, E_N \setminus A_N^x) = o(1),
\ee
and that together with~\eqref{eq-CZoutsidetube} proves the lemma.
\end{proof}

Combining these lemmas, we can now prove Proposition~\ref{prop-capacities-between-metastable-sets-ub}.
\begin{proof}[Proof of Proposition~\ref{prop-capacities-between-metastable-sets-ub}]
Let $S^1 \subsetneq S$ be such that $S_\star^1 \subseteq S^1$ and $S_\star^2 \subseteq S\setminus S^1 =:S^2$.
Note that if $\eta\in \mathcal{E}_N(S_\star^2)$ then $F_{S_1}(\eta)=\col{0}$, and if $\eta\in \mathcal{E}_N(S_\star^1)$ then $F_{S_1}(\eta)=\col{1}$.
Hence, $F_{S_1} \in \mathcal{F}_N(\mathcal{E}(S_\star^1),\mathcal{E}(S_\star^2))$.
Therefore, by~\eqref{eq-capasinf},
\be
\Capac_N(\mathcal{E}(S_\star^1),\mathcal{E}(S_\star^2)) \leq D_N(F_{S^1}).
\ee
We can split the right hand side according to~\eqref{eq-splitD1} and~\eqref{eq-splitD2}, and the proposition then
follows from Lemmas~\ref{lem-insidetubesS1S2},~\ref{lem-insidetubesS1} and~\ref{lem-outsidetubes}.
\end{proof}

\subsection{Proof of Theorem \ref{thm-first-timescale}}\label{sec-Proof-thm1}
\begin{proof}[Proof of Theorem \ref{thm-first-timescale}(i)]
As a consequence of Propositions \ref{prop-capacities-between-metastable-sets-lb} and \ref{prop-capacities-between-metastable-sets-ub},
we have that for nonempty subsets $S_\star^1 \subsetneq S_\star$ and $S_\star^2 = S_\star \setminus S_\star^1$,
and  $d_N \log N \to 0$ as $N\to\infty$,
\be\label{eq-capa1}
\lim_{N\rightarrow\infty} \frac{1}{d_N}
\Capac_N \left(\mathcal{E}_N(S_\star^1),\mathcal{E}_N(S_\star^2)\right)=\frac{1}{\kappa_\star} \sum_{x\in S_\star^1}\sum_{y \in S_\star^2}r(x,y).
\ee
In view of \eqref{eq-potential} and \eqref{eq-hitting-time}, in order to prove the statement (i) we need
to provide an asymptotic formula for the $\mu_N$-average of the equilibrium potential
$\P_\n\left( \t_{\cE_N^x}< \t_{\cE_N(S_\star \setminus \{x\})}\right)$.
Since this is trivially equal to $1$ for $\n\in\cE_N^x$, and equal to $0$ for $\n\in \cE_N(S_\star \setminus \{x\})$,
we have on one hand
\be
\sum_{\n\in E_N} \mu_N(\n)\cdot \P_\n\left( \t_{\cE_N^x}< \t_{\cE_N(S_\star \setminus \{x\})}\right)
\geq \mu_N (\cE_N^x)\,,
\ee
and on the other hand
\be
\sum_{\n\in E_N} \mu_N(\n)\cdot \P_\n\left( \t_{\cE_N^x}< \t_{\cE_N(S_\star \setminus x)}\right)
\leq \sum_{\n\in E_N \atop \n \notin \cE_N(S_\star \setminus \{x\})} \mu_N (\n)= \mu_N (\cE_N^x) + \mu_N(\Delta)\,.
\ee
From these bounds and Proposition \ref{prop-met-sets}, it follows
\be\label{eq-potential-final}
\sum_{\n\in E_N} \mu_N(\n)\cdot \P_\n\left( \t_{\cE_N^x}< \t_{\cE_N(S_\star \setminus \{x\})}\right)
= \frac{1}{\kappa_\star}(1+o(1))\,,
\ee
that together with \eqref{eq-capa1} concludes the proof of \eqref{eq-mean-time}.
\end{proof}
\begin{proof}[Proof of Theorem \ref{thm-first-timescale}(ii)]
We stress once more that in our setting, where metastable sets are just singletons,
the convergence of the speeded-up  process follows from Theorem 2.7 of \cite{BelLan10}
once the condition \eqref{H0}  of Section \ref{subsec-outline} (called condition {\bf(H0)} in \cite{BelLan10}) is verified
for the sequence $\theta_N= 1/d_N$, $N\geq 1$.

By Lemma 6.8 of \cite{BelLan10}, that we have recalled in \eqref{eq-rate-capacity}, and using
Proposition \ref{prop-met-sets} and \eqref{eq-capa1}, we get that for any $x,y \in S_\star$, $x\neq y$,
\be
\lim_{N\to\infty} \frac{1}{d_N} r_N^{\cE_\star}\left(\cE_N^x,\cE_N^y\right)
=  r(x,y)\,.
\ee
To prove \eqref{eq-0time}  observe that by the stationarity of $\mu_N$ we have
\begin{equation}
\begin{split}
\mathbb E_{\mathcal{E}_N^x}\left[\int_0^T \1 \{\eta(s/ d_N)\in \Delta\} \,\dint s\right]&
\leq \frac{1}{\mu_N(\mathcal{E}_N^x)} \sum_{\eta \in E_N} \mu_N(\eta) \mathbb{E}_{\eta}\left[\int_0^T \1 \{\eta(s / d_N)\in \Delta\} \,\dint s\right]\\
& = T \cdot \frac{\mu_N(\Delta)}{\mu_N(\mathcal{E}_N^x)}\,.
\end{split}
\end{equation}
Then \eqref{eq-0time} follows from Proposition \ref{prop-met-sets}.
This concludes the proof of theorem.

\end{proof}


\section{Dynamics of the condensate on the second time-scale}\label{sec-cap-second-timescale}
This \col{s}ection is organized similarly to the previous one. We first provide a lower bound on
capacities, then a matching upper bound, and finally we give the proof of Theorem \ref{thm-second-timescale}.
\subsection{Lower bound on capacities}
\begin{proposition}\label{prop-capa2-lb}
Let the underlying random walk  be as in (\ref{linear-dynamics}), with $\kappa=3$.
Then, for $d_N\log_N\to0$ as $N\to\infty$,
\be
\col{\liminf_{N\to\infty}} \frac{N}{d_N^2} \Capac_N \left(\mathcal{E}_N(1),\mathcal{E}_N(3)\right) \geq \left(\frac{1}{r(1,2)}+\frac{1}{r(3,2)}\right)^{-1}\frac{1}{1-m_\star(2)}.
\ee
\end{proposition}
\begin{proof}
Fix a function $F\in\mathcal{F}(\mathcal{E}_N(1),\mathcal{E}_N(3))$. Using reversibility, we can write the Dirichlet form of $F$ as
\begin{align}
D_N(F)&=\!\! \sum_{\eta\in E_N}\!\mu_N(\eta) \biggl( \eta_1(d_N+\eta_2)r(1,2) \left[F(\eta^{1,2})-F(\eta)\right]^2\! + \eta_2(d_N+\eta_3)r(2,3) \left[F(\eta^{2,3})-F(\eta)\right]^2\biggr) \nn\\
&= \sum_{\xi\in E_{N-1}}  \biggl( \mu_N(\xi+\partial_1)(\xi_1+1)(d_N+\xi_2)r(1,2) \left[F(\xi+\partial_2)-F(\xi+\partial_1)\right]^2 \nn\\
&\qquad \ \qquad+ \mu_N(\xi+\partial_2) (\eta_2+1) (d_N+\xi_3)r(2,3) \left[F(\xi+\partial_3)-F(\xi+\partial_2)\right]^2\biggr),
\end{align}
where $\xi+\partial_z$ denotes a configuration $\xi$ with $N-1$ particles, and with one extra particle on $z$.

For some fixed $L$ and $N$ big enough, we can restrict the Dirichet form of $F$
by only considering configurations $\xi\in E_N$ such that $\xi_1=j,\, \xi_2=\ell$ and $\xi_3=N-j-\ell-1$,
with $\ell\leq L$. On this set of configurations, we then define the function $G(j,\ell,z):=F(\xi+\partial_z) $ and write
\begin{align}\label{banana}
D_N(F) &\geq \frac{1}{Z_{N,S}}  \sum_{\ell=0}^L \sum_{j=0}^{N-\ell-1} \nn\\
&\,\, \biggl\{w_N(j+1)m_\star(2)^{\ell}w_N(\ell)w_N(N-j-\ell-1)(j+1)(d_N+\ell)r(1,2) [G(j,\ell,2)-G(j,\ell,1)]^2 \nn\\
&\qquad \ \qquad + w_N(j) m_\star(2)^{\ell+1}w_N(\ell+1) w_N(N-j-\ell-1) (\ell+1) (d_N+N-j-\ell-1) \nn\\
&\qquad \ \qquad \qquad\quad \cdot r(2,3) [G(j,\ell,3)-G(j,\ell,2)]^2\biggr\} \, \col{,}
\end{align}
\col{w}here we used that $m_\star(2) r(2,3)=r(3,2)$ by the reversibility of the underlying random walk.
From inequality~\eqref{eq-bounddnsq}, we can then bound \eqref{banana} from below by
\be
\begin{split}
\frac{d_N^2}{Z_{N,S}} &\sum_{\ell=0}^L  m_\star(2)^{\ell} \sum_{j=0}^{N-\ell-1} \biggl\{w_N(N-j-\ell-1)r(1,2) [G(j,\ell,2)-G(j,\ell,1)]^2
\\
&\quad + w_N(j) r(3,2) [G(j,\ell,3)-G(j,\ell,2)]^2\biggr\}\,.
\end{split}
\ee

Moreover, let us define
\be
\tilde{w}_N(j) = \left\{\begin{array}{ll}  d_N, & {\rm if\ } j=0, \\ w_N(j),& {\rm if\ } j>0, \end{array}\right.
\ee
so that  $w_N(0) = 1 = \tilde{w}_N(0) + (1-d_N)$ and hence
\begin{align}
D_N(F) &\geq \frac{d_N^2}{Z_{N,S}} \sum_{\ell=0}^L  m_\star(2)^{\ell} \sum_{j=0}^{N-\ell-1} \biggl\{\tilde{w}_N(N-j-\ell-1)r(1,2) [G(j,\ell,2)-G(j,\ell,1)]^2 \nn\\
&\qquad \ \qquad \ \qquad + \tilde{w}_N(j) r(3,2) [G(j,\ell,3)-G(j,\ell,2)]^2\biggr\} \nn\\
&\qquad+ (1-d_N) \frac{d_N^2}{Z_{N,S}} \sum_{\ell=0}^{L-1}  m_\star(2)^{\ell}\biggl( r(1,2) [G(N-\ell-1,\ell,2)-G(N-\ell-1,\ell,1)]^2 \nn\\
&\qquad \ \qquad \ \qquad + r(3,2) [G(0,\ell,3)-G(0,\ell,2)]^2\biggr)\,.
\end{align}
Using Lemma~\ref{lem-eff-resistance} with
\be
g(z) = \frac{G(j,\ell,z)-G(j,\ell,3)}{G(j,\ell,1)-G(j,\ell,3)},
\ee
we can bound
\begin{align}
\tilde{w}_N&(N-j-\ell-1)r(1,2) [G(j,\ell,2)-G(j,\ell,1)]^2+ \tilde{w}_N(j) r(3,2) [G(j,\ell,3)-G(j,\ell,2)]^2 \nn\\
&\geq [G(j,\ell,1)-G(j,\ell,3)]^2 \left( \frac{1}{\tilde{w}_N(N-j-\ell-1)r(1,2)}+\frac{1}{\tilde{w}_N(j) r(3,2)}\right)^{-1}.
\end{align}
Observing that $G(j,\ell,1)=G(j+1,\ell,3)$, and using Lemma~\ref{lem-eff-resistance} again, we bound
\begin{align}
\sum_{j=0}^{N-\ell-1} &[G(j,\ell,1)-G(j,\ell,3)]^2
\left( \frac{1}{\tilde{w}_N(N-j-\ell-1)r(1,2)}+\frac{1}{\tilde{w}_N(j) r(3,2)}\right)^{-1} \\
&\geq [G(N-\ell,\ell,3)-G(0,\ell,3)]^2 \left(\sum_{j=0}^{N-\ell-1}
 \left( \frac{1}{\tilde{w}_N(N-j-\ell-1)r(1,2)}+\frac{1}{\tilde{w}_N(j) r(3,2)}\right)\right)^{-1}.\nn
\end{align}
By reversing the summing order of the first term, the sum over $j$ equals
\begin{align}
\left( \frac{1}{r(1,2)}+\frac{1}{r(3,2)}\right)
\sum_{j=0}^{N-\ell-1}  \frac{1}{\tilde{w}_N(j)}
 &= \left( \frac{1}{r(1,2)}+\frac{1}{r(3,2)}\right)
 \frac{1}{d_N}\left(1+\sum_{j=1}^{N-\ell-1} j(1+o(1))\right) \nn\\
&= \left( \frac{1}{r(1,2)}+\frac{1}{r(3,2)}\right) \frac{N^2}{2d_N}(1+o(1)),
\end{align}
since $\ell=o(N)$. Hence,
\begin{align}
D_N(F) &\geq \frac{d_N^2}{Z_{N,S}}
\sum_{\ell=0}^L  m_\star(2)^{\ell}  [G(N-\ell,\ell,3)-G(0,\ell,3)]^2
\left(\left( \frac{1}{r(1,2)}+\frac{1}{r(3,2)}\right) \frac{N^2}{2d_N}\right)^{-1}(1+o(1)) \nn\\
&\qquad+(1-d_N) \frac{d_N^2}{Z_{N,S}}
\sum_{\ell=0}^{L-1}  m_\star(2)^{\ell}\biggl( r(1,2) [G(N-\ell-1,\ell,2)-G(N-\ell-1,\ell,1)]^2 \nn\\
&\qquad \ \qquad \ \qquad +  r(3,2) [G(0,\ell,3)-G(0,\ell,2)]^2\biggr) \nn\\
&= \frac{d_N^2}{Z_{N,S}}  \sum_{\ell=0}^L  \biggl\{ m_\star(2)^{\ell} \frac{2d_N}{N^2}
\left( \frac{1}{r(1,2)}+\frac{1}{r(3,2)}\right)^{-1}(1+o(1)) [G(N-\ell,\ell,3)-G(0,\ell,3)]^2  \nn\\
&\qquad+(1-d_N) \sum_{p=0}^{\ell-1}  \frac{m_\star(2)^{p}}{L-p}\biggl( r(1,2) [G(N-p-1,p+1,3)-G(N-p,p,3)]^2 \nn\\
&\qquad \ \qquad \ \qquad +  r(3,2) [G(0,p,3)-G(0,p+1,3)]^2\biggr) \biggr\}.
\end{align}
Using Lemma~\ref{lem-eff-resistance}, we can bound
\begin{align}
\sum_{p=0}^{\ell-1}  \frac{m_\star(2)^{p}}{L-p} &[G(N-p-1,p+1,3)-G(N-p,p,3)]^2 \nn\\
&\geq  [G(N-\ell,\ell,3)-G(N,0,3)]^2 \left(\sum_{p=0}^{\ell-1}  \frac{L-p}{m_\star(2)^{p}}\right)^{-1} \nn\\
&\geq  [G(N-\ell,\ell,3)-G(N,0,3)]^2  \frac{m_\star(2)^\ell}{L^2},
\end{align}
and
\be
\sum_{p=0}^{\ell-1}  \frac{m_\star(2)^{p}}{L-p} [G(0,p,3)-G(0,p+1,3)]^2 \geq [G(0,0,3)-G(0,\ell,3)]^2 \frac{m_\star(2)^\ell}{L^2}.
\ee
Thus,
\begin{align}
D_N(F) & \geq \frac{d_N^2}{Z_{N,S}}  \sum_{\ell=0}^L  m_\star(2)^{\ell}
\biggl\{  \frac{2d_N}{N^2} \left( \frac{1}{r(1,2)}+\frac{1}{r(3,2)}\right)^{-1}(1+o(1)) [G(N-\ell,\ell,3)-G(0,\ell,3)]^2 \nn \\
&\,\,+ r(1,2)  \frac{1-d_N}{L^2} [G(N-\ell,\ell,3)-G(N,0,3)]^2 +  r(3,2) \frac{1-d_N}{L^2} [G(0,0,3)-G(0,\ell,3)]^2\biggr) \biggr\} \col{,}
\end{align}
and since $G(0,0,3)=0$ and $G(N,0,3)=1$, we get
\begin{align}
D_N(F) & \geq  \frac{d_N^2}{Z_{N,S}}
 \sum_{\ell=0}^L  m_\star(2)^{\ell} \biggl\{  \frac{N^2}{2d_N}\left( \frac{1}{r(1,2)}
 +\frac{1}{r(3,2)}\right)(1+o(1)) +   \frac{r(1,2)L^2}{1-d_N}
 + \frac{r(3,2)L^2}{1-d_N} \bigg\}^{-1} \nn\\
&= \frac{d_N^2}{N} \frac{d_N}{N Z_{N,S}} \left( \frac{1}{r(1,2)}+\frac{1}{r(3,2)}\right)^{-1} (1+o(1))
\sum_{\ell=0}^L  m_\star(2)^{\ell}\,.\end{align}
By Proposition~\ref{prop-limZN}, $\lim_{N\to\infty} \frac{d_N}{N Z_{N,S}}  = \frac{1}{\kappa_\star}=\frac12$. Hence,
\be
\col{\liminf_{N\to\infty}}\frac{N}{d_N^2} D_N(F) \geq \left(\frac{1}{r(1,2)}+\frac{1}{r(3,2)}\right)^{-1} \sum_{\ell=0}^L m_\star(2)^\ell,
\ee
and the proposition follows by taking $L\to\infty$.
\end{proof}

\begin{remark}
The above lemma can be generalized to systems with arbitrary set $S$ and underlying dynamics,
such that $S_\star=\{x,y\}$ with $x,y$ sites at graph-distance 2. In that case, we have the lower bound
\be
\label{masha}
\col{\liminf_{N\to\infty}} \frac{N}{d_N^2} \Capac_N \left(\mathcal{E}_N(x),\mathcal{E}_N(y)\right)
 \geq  \sum_{v\in S\setminus\{x,y\}} \left(\frac{1}{r(x,v)}+\frac{1}{r(y,v)}\right)^{-1}\frac{1}{1-m_\star(v)}.
\ee
This can easily be proved by restricting the Dirichlet form to those jumps
tha\col{t h}ave at most one vertex $v\in S \setminus S_\star$ with a positive number of particles, and then proceeding as above.
Notice that if it  does not exists $v\in S$ such that $r(x,v)>0$ and $r(y,v)>0$
then the r.h.s. of \eqref{masha} is zero, suggesting the existence of an additional (larger)
time\col{-}scale.
\end{remark}

\subsection{Upper bound on capacities}

\begin{proposition}\label{prop-capa2-ub}
Let the underlying random walk  be as in (\ref{linear-dynamics}), with $\kappa=3$.
Furthermore, suppose that $d_N$ decays subexponentially and $d_N\log N \to 0$ as $N\to\infty$.
 Then,
\be
\col{\limsup_{N\to\infty}} \frac{N}{d_N^2} \Capac_N \left(\mathcal{E}_N(1),\mathcal{E}_N(3)\right) \leq \left(\frac{1}{r(1,2)}+\frac{1}{r(3,2)}\right)^{-1}\frac{1}{1-m_\star(2)}.
\ee
\end{proposition}
\begin{proof}
Since there are only three sites, the space $E_N$ is parameterized
by the number of particles on $1$ and $2$. Let
\be
G(j,\ell) = 2 \left(\frac{1}{r(1,2)}+\frac{1}{r(3,2)}\right)^{-1} \biggl( \frac{1}{r(1,2)} \int_0^{\phi_{2\varepsilon}\left(\frac{j-1}{N}\right)}(1-x) \,\dint x +  \frac{1}{r(3,2)} \int_0^{\phi_{2\varepsilon}\left(\frac{j-1}{N}+(\frac{\ell}{N}\wedge \varepsilon)\right)} x\, \dint x  \biggr),
\ee
and consider the test function
\be
F(\eta) = G(\eta_1,\eta_2).
\ee
We then have that $G(N,0)=1$ and $G(0,0)=0$, so that $F \in \mathcal{F}_N\left(\mathcal{E}_N(1),\mathcal{E}_N(3)\right)$.

Using reversibility we can write the Dirichlet form of $F$ as
\begin{align}
D_N(F) &=\! \sum_{\eta\in E_N}\! \mu_N(\eta) \biggl( \eta_1(d_N+\eta_2)r(1,2) \left[F(\eta^{1,2})-F(\eta)\right]^2\! + \eta_2(d_N+\eta_3)r(2,3) \left[F(\eta^{2,3})-F(\eta)\right]^2\biggr) \nn\\
&= \sum_{\xi\in E_{N-1}}  \biggl( \mu_N(\xi+\partial_1)(\xi_1+1)(d_N+\xi_2)r(1,2) \left[F(\xi+\partial_2)-F(\xi+\partial_1)\right]^2 \nn\\
&\qquad \ \qquad+ \mu_N(\xi+\partial_2) (\eta_2+1) (d_N+\xi_3)r(2,3) \left[F(\xi+\partial_3)-F(\xi+\partial_2)\right]^2\biggr) \nn\\
&=\frac{1}{Z_{N,S}} \sum_{\ell=0}^{N-1} \sum_{j=0}^{N-\ell-1}  \biggl( w_N(j+1)m_\star(2)^\ell w_N(\ell) w_N(N-j-\ell-1)(j+1)(d_N+\ell)r(1,2) \nn\\
& \qquad \ \qquad \ \qquad \ \qquad \ \qquad \ \qquad  \left[G(j,\ell+1)-G(j+1,\ell)\right]^2 \nn\\
&\qquad \ \quad+ w_N(j)m_\star(2)^{\ell+1} w_N(\ell+1) w_N(N-j-\ell-1) (\ell+1) (d_N+N-j-\ell-1)r(2,3) \nn\\
&\qquad \ \qquad \ \qquad \ \qquad \ \qquad \quad  \left[G(j,\ell)-G(j,\ell+1)\right]^2\biggr).
\end{align}
By the definition of $G$, we can compute
\be \label{eq-Gmovefromxtov}
G(j+1,\ell)-G(j,\ell+1) = 2 \left(\frac{1}{r(1,2)}+\frac{1}{r(3,2)}\right)^{-1} \frac{1}{r(1,2)} \int_{\phi_{2\varepsilon}\left(\frac{j-1}{N}\right)}^{\phi_{2\varepsilon}\left(\frac{j}{N}\right)}(1-x) \,\dint x,
\ee
which is $0$ for $j\leq 2\varepsilon N$, and $j> (1-2 \varepsilon) N$. Also
\be \label{eq-Gmovefromvtoy}
G(j,\ell+1)-G(j,\ell) = 2 \left(\frac{1}{r(1,2)}+\frac{1}{r(3,2)}\right)^{-1}  \frac{1}{r(3,2)} \int_{\phi_{2\varepsilon}\left(\frac{j-1}{N}+(\frac{\ell}{N}\wedge \varepsilon)\right)}^{\phi_{2\varepsilon}\left(\frac{j-1}{N}+(\frac{\ell+1}{N}\wedge \varepsilon)\right)} x\, \dint x,
\ee
which is $0$ for $\ell\geq \varepsilon N$, and also for $j\leq \varepsilon N$ or $j> (1-2 \varepsilon) N$. Hence, by Lemma~\ref{lem-asymp-wN},
\begin{align}\label{eq-reqwiteD2nd}
D_N(F) &=\frac{d_N^3 (1+o(1))}{Z_{N,S}}\sum_{\ell=0}^{\varepsilon N} m_\star(2)^\ell  \biggl( \sum_{j=2\varepsilon N}^{ (1-2\varepsilon)N}  \frac{1}{N-j-\ell-1} r(1,2)  \left[G(j,\ell+1)-G(j+1,\ell)\right]^2 \nn\\
&+ \sum_{j=\varepsilon N}^{(1-2\varepsilon)N} \frac{1}{j} r(3,2) \left[G(j,\ell)-G(j,\ell+1)\right]^2\biggr)\\
& + \frac{d_N^2 (1+o(1))}{Z_{N,S}}\sum_{\ell=\varepsilon N+1}^{N-1} m_\star(2)^\ell  \sum_{j=\varepsilon N}^{ N-\ell-1} w_N(N-j-\ell-1) r(1,2)  \left[G(j,\ell+1)-G(j+1,\ell)\right]^2,\nn
\end{align}
where we also used the reversibility of the underlying random walk to substitute ${m_\star(2)r(2,3)=r(3,2)}$.

By~\eqref{eq-Gmovefromxtov}, and for $\ell\leq \varepsilon N$, we have
\begin{align}
\sum_{j=2\varepsilon N}^{ (1-2\varepsilon)N}  & \frac{1}{N-j-\ell-1} r(1,2)  \left[G(j,\ell+1)-G(j+1,\ell)\right]^2 \nn\\
& =\sum_{j=2\varepsilon N}^{ (1-2\varepsilon)N}   \frac{1}{N-j-\ell-1} r(1,2)  \left[ 2 \left(\frac{1}{r(1,2)}+\frac{1}{r(3,2)}\right)^{-1} \frac{1}{r(1,2)} \int_{\phi_{2\varepsilon}\left(\frac{j-1}{N}\right)}^{\phi_{2\varepsilon}\left(\frac{j}{N}\right)}(1-x) \,\dint x \right]^2 \nn\\
&= 4 \left(\frac{1}{r(1,2)}+\frac{1}{r(3,2)}\right)^{-2} \frac{1}{r(1,2)} \sum_{j=2\varepsilon N}^{ (1-2\varepsilon)N} \int_{\phi_{2\varepsilon}\left(\frac{j-1}{N}\right)}^{\phi_{2\varepsilon}\left(\frac{j}{N}\right)}(1-x) \,\dint x \frac{1}{N} \int_{\phi_{2\varepsilon}\left(\frac{j-1}{N}\right)}^{\phi_{2\varepsilon}\left(\frac{j}{N}\right)}\frac{1-x}{1-\frac{j+\ell+1}{N}} \,\dint x.
\end{align}
Then, by the properties of $\phi_{2\varepsilon}$
and using  that $\frac{\ell+2}{N}\leq 2\varepsilon$ for $N$ big enough, 
we get
\be
\int_{\phi_{2\varepsilon}\left(\sfrac{j-1}{N}\right)}^{\phi_{2\varepsilon}\left(\sfrac{j}{N}\right)}\frac{1-x}{1-\frac{j+\ell+1}{N}} \,\dint x \leq \left(\phi_{2\varepsilon}\left(\sfrac{j}{N}\right)-\phi_{2\varepsilon}\left(\sfrac{j-1}{N}\right) \right)\frac{1-\phi_{2\varepsilon}\left(\sfrac{j-1}{N}\right)}{1-\sfrac{j+\ell+1}{N}} \leq (1+\sqrt{\varepsilon})\frac{\phi_{2\varepsilon}\left(1-\sfrac{j-1}{N}\right)}{1-\frac{j-1}{N}-2\varepsilon}\,.
\ee
Using the fundamental theorem of calculus, and that $\phi_{2\varepsilon}(2\varepsilon)=0$,
\be
\phi_{2\varepsilon}\left(1-\sfrac{j-1}{N}\right) = \int_{2\varepsilon}^{1-\frac{j-1}{N}} \phi'_{2\varepsilon}(x) \, \dint x \leq \left(1-\frac{j-1}{N}-2\varepsilon \right)(1+\sqrt{\varepsilon}).
\ee
Hence,
\be
\int_{\phi_{2\varepsilon}\left(\frac{j-1}{N}\right)}^{\phi_{2\varepsilon}\left(\frac{j}{N}\right)}\frac{1-x}{1-\frac{j+\ell+1}{N}} \,\dint x \leq \frac{1}{N}(1+\sqrt{\varepsilon})^2,
\ee
so that
\begin{align}
\sum_{j=2\varepsilon N}^{ (1-2\varepsilon)N} & \frac{1}{N-j-\ell-1} r(1,2)  \left[G(j,\ell+1)-G(j+1,\ell)\right]^2 \nn\\
& \leq \frac{4(1+\sqrt{\varepsilon})^2}{N^2}  \left(\frac{1}{r(1,2)}+\frac{1}{r(3,2)}\right)^{-2} \frac{1}{r(1,2)}  \sum_{j=2\varepsilon N}^{ (1-2\varepsilon)N} \int_{\phi_{2\varepsilon}\left(\frac{j-1}{N}\right)}^{\phi_{2\varepsilon}\left(\frac{j}{N}\right)}(1-x) \,\dint x  \nn\\
& = \frac{2(1+\sqrt{\varepsilon})^2}{N^2}  \left(\frac{1}{r(1,2)}+\frac{1}{r(3,2)}\right)^{-2} \frac{1}{r(1,2)} .
\end{align}
Similarly, we can use~\eqref{eq-Gmovefromvtoy} to bound, for $\ell\leq \varepsilon N$,
\begin{align}
\sum_{j=\varepsilon N}^{(1-2\varepsilon)N} & \frac{1}{j} r(3,2) \left[G(j,\ell)-G(j,\ell+1)\right]^2 \nn\\
&=4 \left(\frac{1}{r(1,2)}+\frac{1}{r(3,2)}\right)^{-2}  \frac{1}{r(3,2)}
\sum_{j=\varepsilon N}^{(1-2\varepsilon)N}\int_{\phi_{2\varepsilon}\left(\frac{j+\ell-1}{N}\right)}^{\phi_{2\varepsilon}\left(\frac{j+\ell}{N}\right)} x\, \dint x \frac{1}{N}\int_{\phi_{2\varepsilon}\left(\frac{j+\ell-1}{N}\right)}^{\phi_{2\varepsilon}\left(\frac{j+\ell}{N}\right)} \frac{x}{j/N}\, \dint x \nn\\
& \leq \frac{2(1+\sqrt{\varepsilon})^2}{N^2}  \left(\frac{1}{r(1,2)}+\frac{1}{r(3,2)}\right)^{-2} \frac{1}{r(3,2)} .
\end{align}
To bound the third line of~\eqref{eq-reqwiteD2nd}, notice that
$|G(j,\ell+1)-G(j+1,\ell)|\leq 1$ and $m_\star(2)^{\ell}\leq m_\star(2)^{\varepsilon N}$, so that
\begin{align}
\sum_{\ell=\varepsilon N+1}^{N-1} &m_\star(2)^\ell
 \sum_{j=\varepsilon N}^{ N-\ell-1} w_N(N-j-\ell-1)  \left[G(j,\ell+1)-G(j+1,\ell)\right]^2 \\
&\leq m_\star(2)^{\varepsilon N} \sum_{\ell=\varepsilon N+1}^{N-1}
\left(1+  \sum_{j=\varepsilon N}^{ N-\ell-2}\frac{d_N(1+o(1))}{N-j-\ell-1} \right)
\leq m_\star(2)^{\varepsilon N} N (1+d_N\log N(1+o(1))).\nn
\end{align}
Hence, we obtain
\begin{align}
D_N(F) &\leq \frac{d_N^3 (1+o(1))}{N^2 Z_{N,S}} 2(1+\sqrt{\varepsilon})^2 \left(\frac{1}{r(1,2)}+\frac{1}{r(3,2)}\right)^{-1} \sum_{\ell=0}^{\varepsilon N} m_\star(2)^\ell   \nn\\
&\qquad \ \qquad+ \frac{d_N^2 (1+o(1))}{Z_{N,S}} m_\star(2)^{\varepsilon N} N (1+d_N\log N).
\end{align}
Taking the limit $N\to\infty$ gives
\begin{align}
\col{\limsup_{N\to\infty}} \frac{N}{d_N^2}D_N(F) &\leq \lim_{N\to\infty}  \frac{d_N (1+o(1))}{N Z_{N,S}} 2(1+\sqrt{\varepsilon})^2 \left(\frac{1}{r(1,2)}+\frac{1}{r(3,2)}\right)^{-1} \sum_{\ell=0}^{\varepsilon N} m_\star(2)^\ell  \nn\\
&\qquad \ \qquad + \frac{d_N (1+o(1))}{NZ_{N,S}} \frac{1}{d_N}m_\star(2)^{\varepsilon N} N^2 (1+d_N\log N) \nn\\
&=\frac{2(1+\sqrt{\varepsilon})^2}{\kappa_\star} \left(\frac{1}{r(1,2)}+\frac{1}{r(3,2)}\right)^{-1} \frac{1}{1-m_\star(2)},
\end{align}
where we used that $d_N$ decays subexponentially to show that the second part converges to $0$.
The proposition follows by taking the limit $\varepsilon\to 0$ and noting that $\kappa_\star=2$.
\end{proof}

\subsection{Proof of Theorem \ref{thm-second-timescale}}\label{sec-Proof-thm2}
The proof runs similarly to that of Theorem  \ref{thm-first-timescale}.
\begin{proof}[Proof of \col{Theorem} \ref{thm-second-timescale}(i)]
As a consequence of Propositions \ref{prop-capa2-lb} and \ref{prop-capa2-ub},
if $d_N$ decays subexponentially and $d_N\log N \to 0$ as $N\to\infty$,
\be\label{eq-capa2}
\lim_{N\to\infty} \frac{N}{d_N^2} \Capac_N \left(\mathcal{E}_N(1),\mathcal{E}_N(3)\right)
= \left(\frac{1}{r(1,2)}+\frac{1}{r(3,2)}\right)^{-1}\frac{1}{1-m_\star(2)}\,.
\ee
In view of \eqref{eq-hitting-time}, recalling that $\Capac_N(A,B)=\Capac_N(B,A)$,
and applying  \eqref{eq-potential-final}, this provides  formula \eqref{eq-mean-time2}.
\end{proof}
\begin{proof}[Proof of \col{Theorem} \ref{thm-second-timescale}(ii)]
As in the proof of \col{Theorem} \ref{thm-first-timescale}(ii),
the convergence follows from Theorem 2.7 of \cite{BelLan10}
once condition \eqref{H0} of Section \ref{subsec-outline} is verified
for the sequence $\theta_N= N/d_N^2$, $N\geq 1$.
By Lemma 6.8 of \cite{BelLan10} (see \eqref{eq-rate-capacity} in Section \ref{subsec-outline})
and using Proposition \ref{prop-met-sets} and \eqref{eq-capa2}, we get
\be
\lim_{N\to\infty} \frac{N}{d_N^2} r_N^{\cE_\star}\left(\cE_N^1,\cE_N^3\right)
= \lim_{N\to\infty} \frac{N}{d_N^2} r_N^{\cE_\star}\left(\cE_N^3,\cE_N^1\right)
=   \left(\frac{1}{r(1,2)}+\frac{1}{r(3,2)}\right)^{-1}\frac{1}{1-m_\star(2)}\,,
\ee
proving \eqref{H0}.
Finally, \eqref{eq-0time2} is proved similarly to \eqref{eq-0time}.
\end{proof}

\section{Dynamics of the condensate on the third time-scale}\label{sec-cap-third-timescale}
In this last \col{s}ection we study the third time-scale that appears when the
condensate moves between sites in $S_{\star}$ that are at graph-distance
larger than 2.
\subsection{Lower bound on capacities}
\label{Lower-bound-third}

\begin{proposition}\label{prop-capa3-lb}
Let the underlying random walk  be as in (\ref{linear-dynamics}), with $\kappa\geq 4$.
Then, for $d_N\log_N\to0$ as $N\to\infty$,
\be\label{eq-lbu-capa3}
\col{\liminf_{N\to\infty}} \frac{N^2}{d_N^3} \Capac_N \left(\mathcal{E}_N(1),\mathcal{E}_N(\kappa)\right) \geq 3 \biggl( \sum_{\col{p}=2}^{\kappa-2}\frac{1}{m_\star(\col{p}) r(\col{p},\col{p}+1)} \biggr)^{-1}.
\ee
\end{proposition}
\begin{proof}
This lower bound is given by transporting particles, one at a time, from $1$ to $\kappa$.
To see this, consider any function $F:E_N\mapsto \R$ such that $F(\eta_1=N)=1$ and $F(\eta_\kappa=N)=0$.  We first use reversibility to write
\begin{align}
D_N(F) &= \frac12 \sum_{\eta\in E_N} \mu_N(\eta) \sum_{z,w \in S} \eta_z(d_N+\eta_w) r(z,w) [F(\eta^{z,w})-F(\eta)]^2 \nn\\
&= \sum_{\eta\in E_N} \mu_N(\eta) \sum_{\col{p}=1}^{\kappa-1} \eta_{\col{p}} (d_N+\eta_{\col{p}+1}) r(\col{p},\col{p}+1) [F(\eta^{\col{p},\col{p}+1})-F(\eta)]^2.
\end{align}
We bound this from below by considering only configurations parameterized by $(j,\col{p})$,
 where ${\eta_1=j}, {\eta_\kappa=N-j-1}$, and one extra particle is on the site $\col{p}$ (that may also be $1$ or $\kappa$).
 This gives
\begin{align}
&D_N(F) \geq \sum_{j=0}^{N-1} \sum_{\col{p}=1}^{\kappa-1} \mu_N(j,\col{p}) \eta_\ell (d_N+\eta_{\col{p}+1}) r(\col{p},\col{p}+1)[F(j,\col{p}+1)-F(j,\col{p})]^2 \nn\\
&=\frac{1}{Z_N} \sum_{j=0}^{N-1} \biggl( \sum_{\col{p}=2}^{\kappa-2} w_N(j)w_N(N-j-1) m_\star(\col{p}) w_N(1) d_N r(\col{p},\col{p}+1) [F(j,\col{p}+1)-F(j,\col{p})]^2 \nn\\
&\quad+ w_N(j+1)w_N(N-j-1)(j+1)d_N r(1,2)  [F(j,2)-F(j,1)]^2  \nn\\
&\quad + w_N(j)w_N(N-j-1)m_\star(\kappa-1) w_N(1) (d_N+N-j-1)r(\kappa-1,\kappa) [F(j,\kappa)-F(j,\kappa-1)]^2\biggr) \nn\\
&\geq \frac{d_N^4}{Z_N} \sum_{j=1}^{N-2} \biggl( \sum_{\col{p}=2}^{\kappa-2}\frac{m_\star(\col{p})}{j(N-j-1)}  r(\col{p},\col{p}+1) [F(j,\col{p}+1)-F(j,\col{p})]^2 \nn\\
&\quad+ \frac{1}{d_N(N-j-1)}r(1,2)  [F(j,2)-F(j,1)]^2  + \frac{m_\star(\kappa-1)}{d_N j}r(\kappa-1,\kappa)  [F(j,\kappa)-F(j,\kappa-1)]^2\biggr)\nn\\
& \quad+ \{j=0 {\rm\ term}\}+ \{j=N-1 {\rm\ term}\},
\end{align}
where we used Lemma~\ref{lem-asymp-wN} for the second inequality. By Lemma~\ref{lem-eff-resistance}, we can bound $D_N(F)$ further
\begin{align}
D_N(F) &\geq \frac{d_N^4}{Z_N} \sum_{j=1}^{N-2} [F(j,1)-F(j,\kappa)]^2 \frac{1}{j(N-j-1)} \biggl( \sum_{\col{p}=1}^{\kappa-2}\frac{1}{m_\star(\col{p}) r(\col{p},\col{p}+1)} +o(1)\biggr)^{-1}\nn\\
&\qquad + \{j=0 {\rm\ term}\}+ \{j=N-1 {\rm\ term}\} \nn\\
&\geq \frac{d_N^4}{Z_N} \biggl( \sum_{\col{p}=2}^{\kappa-2}\frac{1}{m_\star(\col{p}) r(\col{p},\col{p}+1)} +o(1)\biggr)^{-1} \biggl(\sum_{j=1}^{N-2} j(N-j-1)\biggr)^{-1},
\end{align}
where we ignored the terms for $j=0$ and $j=N-1$. The sum in the last brackets equals $\frac16 (N-2)(N-1)N$.
Hence,
\be
\frac{N^2}{d_N^3} D_N(F) \geq \frac{d_N}{N Z_N} \col{\frac{6 N^2}{(N-2)(N-1)}}
\biggl( \sum_{\col{p}=2}^{\kappa-2}\frac{1}{m_\star(\col{p}) r(\col{p},\col{p}+1)} +o(1)\biggr)^{-1}\,,
\ee
that in the limit $N\to\infty$ indeed complete\col{s} the proof.
\end{proof}

\begin{remark}
On general graphs, this lower bound on the capacity between sites in $S_\star$
that are at graph distance at least three is also valid, since the Dirichlet form
can always be restricted to only allow for jumps on one specific path,
and then restricting the jumps further as in the proof.
This proves that also in general systems longer time-scales cannot be present.
\end{remark}

\subsection{Upper bound on capacities}
We have the following upper bound:
\begin{proposition}\label{prop-capa3-ub}
Let the underlying random walk  be as in (\ref{linear-dynamics}), with $\kappa\geq 4$.
Furthermore, suppose that $d_N$ decays subexponentially and $d_N\log N \to 0$ as $N\to\infty$.
Then,
\be\label{eq-ubu-capa3}
\col{\limsup_{N\to\infty}} \frac{N^2}{d_N^3} \Capac_N \left(\mathcal{E}_N(1),\mathcal{E}_N(\kappa)\right)
\leq  3\left(\sum_{p=2}^{\kappa-2}\frac{(1-m_\star(p))(1-m_\star(p+1))}{m_\star(p) r(p,p+1)}\right)^{-1}.
\ee
\end{proposition}
\begin{proof}
From the lower bound, we can guess that a good test function should be of the form
\be
F(\eta) =  6 \sum_{\ell=2}^{\kappa-2} c_\ell \int_0^{\phi_{2\varepsilon}\left(\frac{\eta_1}{N}
+ \left((\frac{1}{N}\sum_{p=2}^\ell \eta_p) \wedge \varepsilon\right)\right)}x(1-x) \, \dint x,
\ee
where we need to choose the constants such that
\be\label{eq-normalizationcell}
\sum_{\ell=2}^{\kappa-2}c_\ell = 1,
\ee
so that
\be
F(\eta_1=N) = 6 \sum_{\ell=2}^{\kappa-2}c_\ell \int_0^1 x(1-x) \, \dint x = 1.
\ee
We obviously also have that $F(\eta_\kappa=N)=0$,
 so that $F \in \mathcal{F}_N\left(\mathcal{E}_N(1),\mathcal{E}_N(\kappa)\right)$.
  We optimize over the constants $c_\ell$ at the end.

Because of the choice of $\phi_{2\varepsilon}$,
 we have that $F(\eta^{p,p+1}) - F(\eta) =0$ for all $p=1,\ldots,\kappa-1$ if $j<\varepsilon N$ \col{or} $j>(1-\varepsilon)N$.
 Denote by $\ell$ the total number of particles on sites $2,\ldots,\kappa-1$. Then, we have that, for $\ell < \varepsilon N$,
\be
F(\eta^{1,2}) - F(\eta) =0.
\ee
We also have, for all values of $\ell$, that
\be
F(\eta^{\kappa-1,\kappa}) - F(\eta) =0.
\ee
Hence, also using reversibility in the first equality, we can  rewrite the Dirichlet form of $F$ as,
\begin{align}\label{split-D3-1}
D_N(F) &= \sum_{\eta \in E_N} \mu_N(\eta) \sum_{q=1}^{\kappa-1} \eta_q (\eta_{q+1}+d_N) r(q,q+1) \left[F(\eta^{q,q+1}) - F(\eta) \right]^2 \nn\\
&= \sum_{\ell=0}^{\varepsilon N - 1} \sum_{j=\varepsilon N}^{(1-\varepsilon)N}
\sum_{\eta_2+\ldots+\eta_{\kappa-1}=\ell} \mu_N(\eta) \sum_{q=2}^{\kappa-2}
\eta_q (\eta_{q+1}+d_N) r(q,q+1) \left[F(\eta^{q,q+1}) - F(\eta) \right]^2 \\
&\qquad + \sum_{\ell=\varepsilon N}^{N} \sum_{j=\varepsilon N}^{N-\ell} \sum_{\eta_2+\ldots+\eta_{\kappa-1}=\ell}
 \mu_N(\eta) \sum_{q=1}^{\kappa-2} \eta_q (\eta_{q+1}+d_N) r(q,q+1) \left[F(\eta^{q,q+1}) - F(\eta) \right]^2. \nn
\end{align}
For small $\ell$, we split the sum
\be\label{eq-spliitsum3rd}
 \sum_{\eta_2+\ldots+\eta_{\kappa-1}=\ell} = \sum_{p=2}^{\kappa-2} \sum_{\eta_p=1}^{\ell} \1 \{\eta_{p+1}=\ell-\eta_p\} +
  \sum_{\substack{\eta_2+\ldots+\eta_{\kappa-1}=\ell \\ \eta_p+\eta_{p+1} < \ell \, \forall \,  2\leq p \leq \kappa-2}}.
\ee
The first sum consists of all configurations with $\ell$ particles on at most $2$ adjacent sites in ${\{2,\ldots,\kappa-1\}}$,
and with the rest of the particles only on sites $1$ and $\kappa$, while  the second sum consists
of all other configurations.
This latter sum turns out to have a negligible contribution, as w\col{e s}how later.
Let us first analyze the first sum:
\begin{align}
 \sum_{\ell=0}^{\varepsilon N - 1} &\sum_{j=\varepsilon N}^{(1-\varepsilon)N} \sum_{p=2}^{\kappa-2} \sum_{\eta_p=1}^{\ell} \1 \{\eta_{p+1}=\ell-\eta_p\}  \mu_N(\eta) \sum_{q=2}^{\kappa-2} \eta_q (\eta_{q+1}+d_N) r(q,q+1) \left[F(\eta^{q,q+1}) - F(\eta) \right]^2 \nn\\
&= \frac{1}{Z_N}\sum_{\ell=0}^{\varepsilon N - 1}  \sum_{j=\varepsilon N}^{(1-\varepsilon)N} w_N(j) w_N(N-j-\ell) \sum_{p=2}^{\kappa-2} \sum_{\eta_p=1}^\ell w_N(\eta_p)m_\star(p)^\ell w_N(\ell-\eta_p)m_\star(p+1)^{\ell-\eta_p} \nn\\
&\qquad \ \qquad \ \qquad \ \qquad \ \qquad \ \qquad \sum_{q=p}^{(p+1)\wedge (\kappa-2)} \eta_q (\eta_{q+1}+d_N) r(q,q+1) \left[F(\eta^{q,q+1}) - F(\eta) \right]^2,
\end{align}
since all other $q$ give a $0$ contribution, because then $\eta_q=0$.

Using Lemma~\ref{lem-asymp-wN}\col{,} the above equals
\begin{align}\label{eq-thirdscalesomewhere}
\frac{d_N^4}{Z_N}&(1+o(1))\sum_{\ell=0}^{\varepsilon N - 1} \sum_{j=\varepsilon N}^{(1-\varepsilon)N} \frac{1}{j(N-j-\ell)} \sum_{p=2}^{\kappa-2} \sum_{\eta_p=1}^\ell m_\star(p)^{\eta_p} m_\star(p+1)^{\ell-\eta_p} \nn\\
&\qquad  \left( r(p,p+1)\left[F(\eta^{p,p+1}) - F(\eta) \right]^2 + \frac{d_N}{\eta_p}r(p+1,p+2) \left[F(\eta^{p+1,(p+2)\wedge(\kappa-2)}) - F(\eta) \right]^2 \right) \nn\\
&=6^2 \frac{d_N^4}{Z_N}(1+o(1))\sum_{\ell=0}^{\varepsilon N - 1}   \sum_{p=2}^{\kappa-2} \sum_{\eta_p=1}^\ell m_\star(p)^{\eta_p} m_\star(p+1)^{\ell-\eta_p} \sum_{j=\varepsilon N}^{(1-\varepsilon)N} \frac{1}{j(N-j-\ell)} \nn\\
&\qquad \ \qquad \biggl(r(p,p+1) \left[c_p \int_{\phi_{2\varepsilon}\left(\frac{j+\eta_p-1}{N}\right)}^{\phi_{2\varepsilon}\left(\frac{j+\eta_p}{N}\right)} x(1-x)\, \dint x \right]^2 \nn\\
&\qquad \ \qquad \ \qquad + \frac{d_N}{\eta_p} r(p+1,p+2)\left[c_{p+1} \int_{\phi_{2\varepsilon}\left(\frac{j+\eta_{p+1}-\1\{p<\kappa-2\}}{N}\right)}^{\phi_{2\varepsilon}\left(\frac{j+\eta_{p+1}}{N}\right)} x(1-x)\, \dint x \right]^2 \biggr).
\end{align}
Similarly to the upper bound in the second time-scale, it holds that
\begin{align}
\sum_{j=\varepsilon N}^{(1-\varepsilon)N} \frac{1}{j(N-j-\ell)} &
\left[ \int_{\phi_{2\varepsilon}\left(\frac{j+\eta_p-1}{N}\right)}^{\phi_{2\varepsilon}\left(\frac{j+\eta_p}{N}\right)} x(1-x)\, \dint x \right]^2 \nn\\
&= \frac{1}{N^2} \sum_{j=\varepsilon N}^{(1-\varepsilon)N}
\int_{\phi_{2\varepsilon}\left(\frac{j+\eta_p-1}{N}\right)}^{\phi_{2\varepsilon}\left(\frac{j+\eta_p}{N}\right)} x(1-x)\, \dint x  \int_{\phi_{2\varepsilon}\left(\frac{j+\eta_p-1}{N}\right)}^{\phi_{2\varepsilon}\left(\frac{j+\eta_p}{N}\right)}
\frac{x}{j/N}\frac{1-x}{1-(j+\ell)/N}\, \dint x \nn\\
& \leq \frac{(1+\sqrt{\varepsilon})^3}{N^3}   \int_0^1 x(1-x)\, \dint x  = \frac{(1+\sqrt{\varepsilon})^3}{6 N^3}.
\end{align}
Hence, \eqref{eq-thirdscalesomewhere} is bounded from above by
\be\label{eq-3rdmaincontri}
6 \frac{(1+\sqrt{\varepsilon})^3}{N^3} \frac{d_N^4}{Z_N}  (1+o(1)) \sum_{p=2}^{\kappa-2} c_p^2 r(p,p+1)\sum_{\ell=0}^{\varepsilon N - 1} \sum_{\eta_p=1}^\ell m_\star(p)^{\eta_p} m_\star(p+1)^{\ell-\eta_p},
\ee
because the contribution of the second part of the last line of~\eqref{eq-thirdscalesomewhere} is clearly $o(1)$ times the contribution of the first part.

To bound the contribution of the last sum in~\eqref{eq-spliitsum3rd},
we set $M=\max_{v\notin S_\star} m_\star(v)$ and observe that, for all such configurations and all $q$,
\be
m_\star^\eta w_N(\eta) \eta_q (\eta_{q+1}+d_N) \leq M^{\ell} \frac{d_N^5}{j(N-j-\ell)},
\ee
because either at least $5$ sites are occupied, or $4$ sites are occupied but $\eta_{q+1}=0$.
Then one can show, as above, that this contribution is also negligible compared to~\eqref{eq-3rdmaincontri}.

To show that the sum over $\ell \geq \varepsilon N$ in~\eqref{split-D3-1} is negligible, we write
\be
\sum_{q=1}^{\kappa-2} \eta_q (\eta_{q+1}+d_N) r(q,q+1) \leq (\kappa-3) R N^2,
\ee
where we set $R=\max_{\ell=2}^{\kappa-2} r(\ell, \ell+1)$.
Furthermore, $ \left[F(\eta^{q,q+1}) - F(\eta) \right]^2 \leq1$ and
\be
\mu_N(\eta) \leq \frac{M^{\varepsilon N}}{Z_N} w_N(\eta) = \frac{M^{\varepsilon N} N }{2 d_N} (1+o(1)) w_N(\eta).
\ee
Hence,
\begin{align}
\sum_{\ell=\varepsilon N}^{N} &\sum_{j=\varepsilon N}^{N-\ell} \sum_{\eta_2+\ldots+\eta_{\kappa-1}=\ell} \mu_N(\eta) \sum_{q=1}^{\kappa-2} \eta_q (\eta_{q+1}+d_N) r(q,q+1) \left[F(\eta^{q,q+1}) - F(\eta) \right]^2 \nn\\
&\leq (\kappa-3) R  \frac{M^{\varepsilon N} N^3 }{2 d_N} (1+o(1)) \sum_{\ell=\varepsilon N}^{N} \sum_{j=\varepsilon N}^{N-\ell} \sum_{\eta_2+\ldots+\eta_{\kappa-1}=\ell} w_N(\eta).
\end{align}
Now we can write
\be
\sum_{\ell=\varepsilon N}^{N} \sum_{j=\varepsilon N}^{N-\ell} \sum_{\eta_2+\ldots+\eta_{\kappa-1}=\ell} w_N(\eta)  \leq \sum_{\ell=0}^{N} \sum_{j=0}^{N-\ell} \sum_{\eta_2+\ldots+\eta_{\kappa-1}=\ell} w_N(\eta)  = \tilde{Z}_{N},
\ee
where $\tilde{Z}_N$ is the partition function of a similar system where we set $m_\star(v)=1$ for all  $v\in\{1,\ldots,\kappa\}$. Hence,
\be
\tilde{Z}_N = \frac{\kappa d_N}{N}(1+o(1)).
\ee
and we get
\begin{align}
\frac{N^2}{d_N^3}\sum_{\ell=\varepsilon N}^{N} &\sum_{j=0}^{N-\ell} \sum_{\eta_2+\ldots+\eta_{\kappa-1}=\ell} \mu_N(\eta) \sum_{q=1}^{\kappa-2} \eta_q (\eta_{q+1}+d_N) r(q,q+1) \left[F(\eta^{q,q+1}) - F(\eta) \right]^2 \nn\\
&\leq \kappa(\kappa-3) R  \frac{M^{\varepsilon N} N^4 }{2 d_N^3} (1+o(1)),
\end{align}
which converges to  $0$ because $d_N$ decays subexponentially.

Thus,  the only significant contribution to $D_N(F)$ can be bounded from above by~\eqref{eq-3rdmaincontri},
and altogether we obtain
\begin{align}\label{eq-upperthreeepsilonc}
\col{\limsup_{N\to\infty}} \frac{N^2}{d_N^3} D_N(F) &\leq \col{\limsup_{N\to\infty}} \, 6 (1+\sqrt{\varepsilon})^3\frac{d_N}{N Z_N}  (1+o(1)) \sum_{p=2}^{\kappa-2} c_p^2 r(p,p+1)\sum_{\ell=0}^{\varepsilon N - 1} \sum_{\eta_p=1}^\ell m_\star(p)^{\eta_p} m_\star(p+1)^{\ell-\eta_p} \nn\\
& = 3 (1+\sqrt{\varepsilon})^3 \sum_{p=2}^{\kappa-2} c_p^2 r(p,p+1) m_\star(p) \sum_{\ell=0}^{\infty} \sum_{\eta_p=0}^{\ell-1} m_\star(p)^{\eta_p} m_\star(p+1)^{\ell-\eta_p}\nn\\
&= 3 (1+\sqrt{\varepsilon})^3 \sum_{p=2}^{\kappa-2} c_p^2 \frac{r(p,p+1) m_\star(p)}{(1-m_\star(p))(1-m_\star(p+1))}.
\end{align}
We finally optimize over the constants $c_p$. Let us write $c_p=g(p)-g(p+1)$. By~\eqref{eq-normalizationcell}, we need that
\be
\sum_{p=2}^{\kappa-2} c_p = \sum_{p=2}^{\kappa-2} (g(p)-g(p+1)) = g(2) - g(\kappa-1) = 1,
\ee
and hence, without loss of generality, we can optimize over functions $g$ such that $g(2)=1$ and $g(\kappa-1)=0$.
 Then, taking the infimum  over all such functions $g$, it follows from~\eqref{eq-upperthreeepsilonc} that
\begin{align}
\col{\limsup_{N\to\infty}} \frac{N^2}{d_N^3} D_N(F)  &
\leq \inf_{g:g(2)=1, g(\kappa-2)=0} 3 (1+\sqrt{\varepsilon})^3 \sum_{p=2}^{\kappa-2}
[g(p)-g(p+1)]^2 \frac{r(p,p+1) m_\star(p)}{(1-m_\star(p))(1-m_\star(p+1))} \nn\\
&= 3 (1+\sqrt{\varepsilon})^3 \left( \sum_{p=2}^{\kappa-2} \frac{(1-m_\star(p))(1-m_\star(p+1))}{r(p,p+1) m_\star(p)} \right)^{-1},
\end{align}
because this is again the effective capacity of a linear chain. The proposition now follows by taking $\varepsilon \to 0$.
\end{proof}

\subsection{Proof of Theorem \ref{thm-third-timescale}}\label{sec-Proof-thm3}
\begin{proof}
As a consequence of Propositions \ref{prop-capa3-lb} and \ref{prop-capa3-ub},
if $d_N$ decays subexponentially and ${d_N\log N \to 0}$ as $N\to\infty$,

\be\label{eq-capa3}
C_1 \leq \liminf_{N\to\infty}  \frac{N^2}{d_N^3}  \Capac_N \left(\mathcal{E}_N(1),\mathcal{E}_N(\kappa)\right)
\leq \limsup_{N\to\infty} \frac{N^2}{d_N^3} \Capac_N \left(\mathcal{E}_N(1),\mathcal{E}_N(\kappa)\right)
\leq C_2,
\ee
for some constant $0<C_1,C_2<\infty$.
In view of  \eqref{eq-hitting-time}, recalling that $\Capac_N(A,B)=\Capac_N(B,A)$,
and applying  \eqref{eq-potential-final}, this provides  formula \eqref{eq-3timescale} and conclude
the proof of the theorem.
\end{proof}
\begin{remark}
The constants in~\eqref{eq-lbu-capa3} and~\eqref{eq-ubu-capa3}  do not match, so that we  do not obtain results as in Theorems~\ref{thm-first-timescale}(ii) and~\ref{thm-second-timescale}(ii). We expect that the lower bound can be improved by not restricting the Dirichlet form to configurations with just one particle outside of $S_\star$, but also allowing a small number of particles to make the transition together. Indeed, these are the configurations that contribute to the upper bound. Including these jumps, however, does not result in a linear chain, and therefore the problem is hard to analyze.

Computing the capacity in general systems is an even harder open problem,
since several (possibly intersecting) paths of varying lengths can give a significant contribution to the Dirichlet form.

\end{remark}

\paragraph*{Acknowledgements.}
We thank \col{the} Institute Henri Poincar\'{e} for the hospitality during the trimester ``Disordered
systems, random spatial processes and their applications''.
We acknowledge financial support from the Italian
Research Funding Agency (MIUR) through FIRB project grant n.\ RBFR10N90W.
\col{The work of SD was partially supported by DFG Research Training Group 2131.}


\end{document}